\documentclass[11pt]{article}

\usepackage{amsmath,amssymb,amsthm}
\usepackage{caption}
\usepackage{mathrsfs}
\usepackage{epsfig}
\usepackage{enumitem}
\usepackage{framed}
\usepackage{longtable}
\usepackage{graphicx,epstopdf}
\usepackage{subfigure}
\usepackage{url}
\usepackage{cases}
\usepackage{verbatim}
\usepackage{algpseudocode}
\usepackage{hyperref}
\usepackage{rotating}
\usepackage{multirow}
\usepackage{booktabs}
\usepackage{algorithmicx,algorithm}
\usepackage[font=scriptsize]{caption}
\allowdisplaybreaks[4]

\newcommand{\B}{\mathbb{B}}
\newcommand{\D}{\mathbb{D}}
\newcommand{\E}{\mathbb{E}}
\newcommand{\R}{\mathbb{R}}

\newtheorem{theorem}{Theorem}
\newtheorem{proposition}{Proposition}
\newtheorem{lemma}{Lemma}
\newtheorem{corollary}{Corollary}

\newtheorem{definition}{Definition}

\title{A Decomposition Algorithm for Two-Stage Stochastic Programs\\ with Nonconvex Recourse}

\author{
Hanyang Li
\and Ying Cui\thanks{Department of Industrial and Systems Engineering, University of Minnesota, Minneapolis, MN 55455 (\href{mailto:li002492@umn.edu}{li002492@umn.edu}, \href{mailto:yingcui@umn.edu}{yingcui@umn.edu}). The authors are partially supported by NSF CCF-2153352.}}
\begin{document}
\bibliographystyle{plain}
\maketitle
\begin{abstract}
In this paper, we have studied a decomposition method for solving a class of nonconvex  two-stage stochastic programs, where both the objective and constraints of the second-stage problem are nonlinearly parameterized by the first-stage variables.
Due to the failure of the Clarke regularity of the resulting nonconvex recourse function, classical decomposition approaches such as Benders decomposition and (augmented) Lagrangian-based algorithms cannot be directly generalized to solve such models. By exploring an implicitly convex-concave structure of the recourse function, we introduce a novel decomposition framework based on the so-called partial Moreau envelope. The algorithm successively generates strongly convex quadratic approximations of the recourse function based on the solutions of the  second-stage convex subproblems and adds them to  the first-stage master problem. Convergence under both fixed scenarios and interior samplings is established. Numerical experiments are conducted to demonstrate the effectiveness of the proposed algorithm.

\paragraph{Keywords:}
two-stage stochastic program, nonconvex recourse, decomposition 
\end{abstract}

\section{Introduction}
\label{sec:introduction}
Stochastic programming (SP) is a mathematical framework to model the decision making in the presence of uncertainty \cite{birge2011introduction,shapiro2021lectures}. Two-stage SPs constitute a special class of this paradigm where partial  decisions  have to be made before the observation of the entire information, while the rest decisions are determined after the full information is revealed. 
Most of the existing computational study of the continuous two-stage SPs is devoted to convex problems, especially linear problems \cite{ruszczynski1997decomposition,birge2011introduction, higle2013stochastic, shapiro2021lectures}.  {
Although there is a significant amount of literature that tackles the nonconvexity of the SPs caused by the integrality of the decision variables \cite{bodur2017strengthened, boland2018combining, schultz2003stochastic, caroe1999dual}, the nonconvexity  is usually abandoned and ignored in the algorithmic development of continuous SPs.}

However, there are many emerging  applications in operations research and machine learning that call for complex nonlinear two-stage SP models and computational methods. Let us first introduce the mathematical formulation of such problems before discussing the applications.
The central optimization problem under consideration in this paper takes the following form:
\begin{equation}\label{eq:first_stage}
\displaystyle{
\operatornamewithlimits{\mbox{minimize}}_{x \, \in \, X}
} \;\, \zeta(x) \triangleq \varphi(x) + \mathbb{E}_{\tilde\xi}\left[ \, \psi(x;\tilde\xi) \, \right],
\end{equation}
where  $\psi(x;\xi)$ is the second-stage recourse function that is given by:
\begin{equation}
    \label{eq:second_stage}
  \psi(x;\xi) \, \triangleq \, 
\left\{\begin{array}{ll}
 \displaystyle\operatornamewithlimits{minimum}_{y}  \left[\, f(x,y;\xi) \;\;
    \mbox{subject to}   \;\; G(x,y;\xi)\leq 0 \,\right] & \mbox{if {$x\in \overline X$}} \\[0.1in]
    +\infty & \mbox{if {$x\notin \overline X$}}
    \end{array}\right..
\end{equation}
In the above formulation, {$X$ and $\overline X$ are nonempty convex compact subsets in $\mathbb R^{n_1}$ with $X \subseteq \operatorname{int}(\overline X)$}, $\varphi:\mathbb{R}^{n_1} \to \mathbb{R}$ is a deterministic {convex} function that only depends on the first-stage decision $x$;
 $\tilde\xi: \Omega \rightarrow \Xi$ is a random vector on a probability space $(\Omega, \mathcal F, \mathbb P)$ with $\Xi\subseteq\mathbb{R}^m$ being a measurable closed set; $\xi = \tilde\xi(\omega)$ for some $\omega \in \Omega$ represents a  realization  of  the  random  vector $\tilde\xi$; and $f:\mathbb{R}^{n_1+n_2}\times \Xi\to\mathbb{R}$ and $G\triangleq (g_1, \ldots, g_\ell)^\top:\mathbb{R}^{n_1+n_2}\times \Xi\to\mathbb{R}^\ell$ are two Carath{\'e}odory functions (i.e., $f(\bullet,\bullet;\xi)$ and $G(\bullet,\bullet;\xi)$ are continuous for almost any $\xi\in\Xi$; $f(x,y;\bullet)$ and $G(x,y;\bullet)$ are measurable for any $(x,y) \in \mathbb{R}^{n_1 + n_2}$) that are jointly determined by the first-stage variable $x$ and the second-stage variable $y$. 
{We assume that, for almost any $\xi\in \Xi$, the function $f(\bullet, \bullet; \xi)$ is concave-convex (i.e., $f(\bullet,y;\xi)$ is concave for $y\in \mathbb{R}^{n_2}$ and $f(x,\bullet;\xi)$ is convex for $x\in \mathbb{R}^{n_1}$),  and $g_i(\bullet,\bullet; \xi)$ is jointly convex for each $i=1, \ldots, \ell$. 
}

An example of concave-convex $f(\bullet,\bullet;\xi)$ is a bilinear function $x^\top D(\xi) y$ for some random matrix $D(\xi)\in \mathbb{R}^{n_1\times n_2}$. The above settings notably extend the classical paradigm for continuous two-stage SPs \cite[Chapter 2.3]{shapiro2021lectures} in the following directions:\\[0.08in]
(i) The first-stage variable $x$  not only appears in the constraints  of the second-stage problem, but also in the objective $f$.  The recourse function $\psi(\bullet;\xi)$ is nonconvex since $f(\bullet,\bullet; \xi)$ is not jointly convex. This is fundamentally different from the 
recent papers \cite{guigues2021inexact,guigues2021inexactmirror} that have assumed the joint convexity of $f(\bullet,\bullet;\xi)$ .\\[0.05in]
(ii) Both the objective function $f$ and the constraint map $G$ can be  nonsmooth.\\[0.08in]
These  two features together lead to {\sl a complex nonconvex and nonsmooth recourse function $\psi(\bullet;\xi)$} (see Figure \ref{fig:nonconvex} below), which constitutes the major challenge for designing rigorous and efficient numerical methods to solve problem \eqref{eq:first_stage}. 

\begin{figure}[h]
    \centering
    \subfigure{
    \begin{minipage}[t]{0.3\linewidth}
    \centering
    \includegraphics[width=1\textwidth]{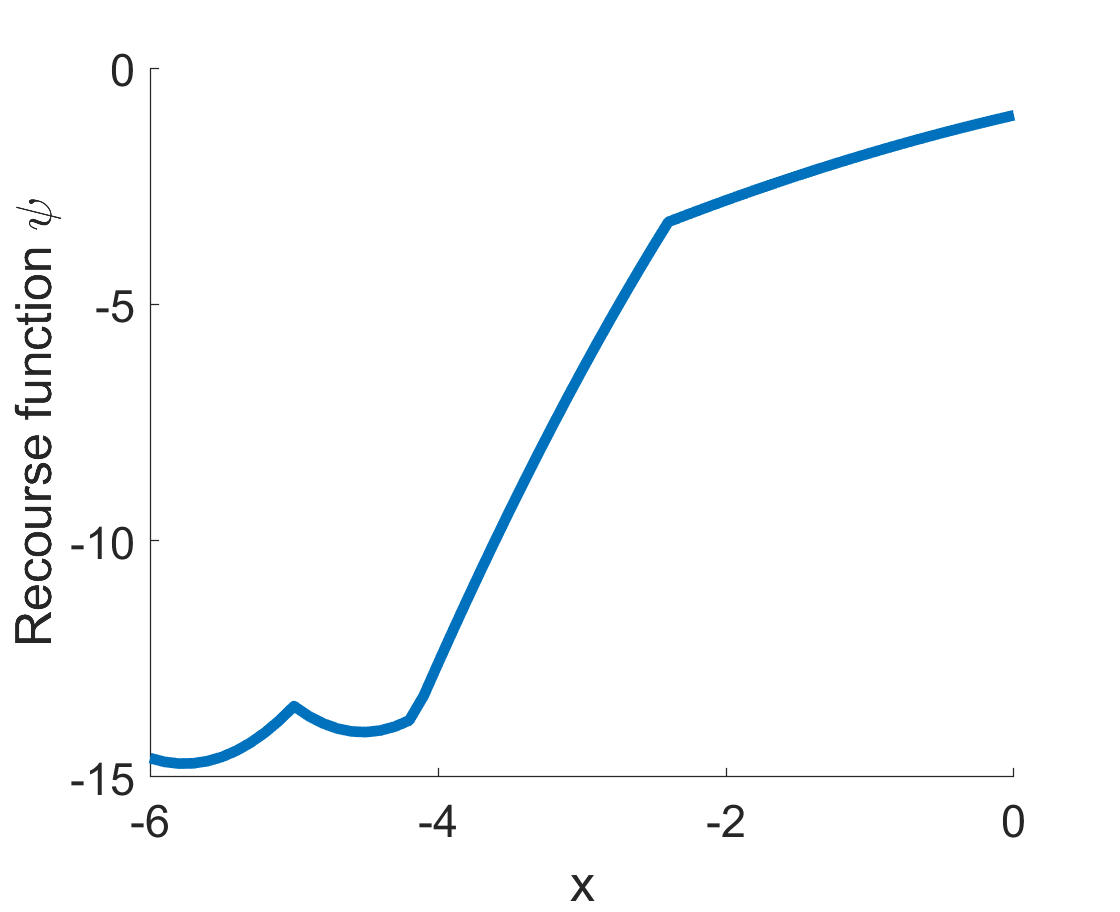}
    \end{minipage}
    }
    \qquad\qquad
    \subfigure{
    \begin{minipage}[t]{0.35\linewidth}
    \centering
    \includegraphics[width=1\textwidth]{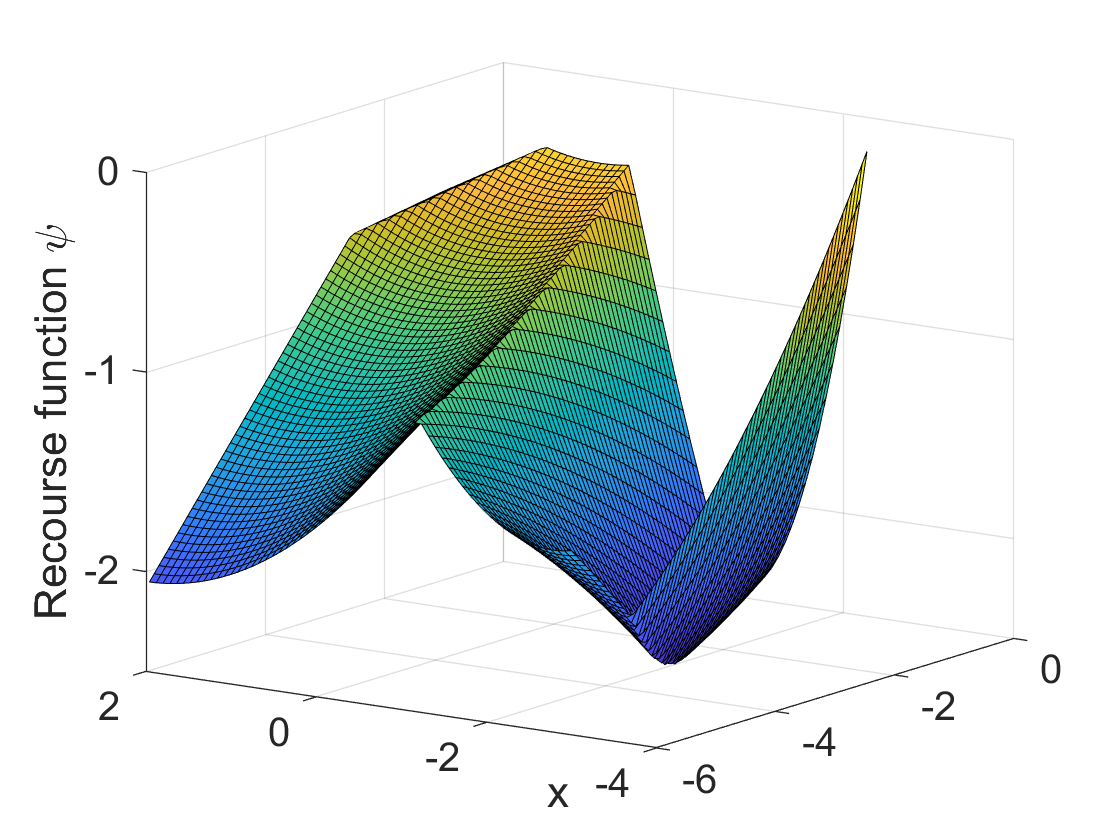}
    \end{minipage}
    }
    \centering
    \caption{The nonconvex nonsmooth recourse functions for fixed $\xi$'s. {\bf Left}: $x\in \mathbb{R}$; {\bf Right}: $x\in \mathbb{R}^2$.}
    \label{fig:nonconvex}
\end{figure}

Recourse functions in the form of \eqref{eq:second_stage} arise from many applications. One important source of the nonconvex recourse in \eqref{eq:second_stage} comes from the decision dependent/influenced uncertainty \cite{hellemo2018decision,nohadani2018optimization, jonsbraaten1998class, goel2004stochastic}, 
where the probability distribution $\mathbb{P}(x)$ of the random vector $\tilde\xi$  is dependent on the first-stage variable $x$. 
This is in contrast to the classical 
SP paradigm under exogenous uncertainty, where the distribution of $\tilde\xi$ is not affected by the first-stage decisions.
There is growing interest in the endogenous uncertainty in recent literature of stochastic and robust programs \cite{goel2005lagrangean,goel2006class, tarhan2009stochastic,hellemo2018decision}.
A typical example where the random parameters can be altered by a decision is that the price (as a first-stage variable) may affect the distribution of the product demand. 
{Assume that the probability distribution of $\tilde\xi$ is given by $\mathbb P_x$ that depends on the first-stage variable $x$. We consider the corresponding two-stage SP model:}
\[
\displaystyle{
\operatornamewithlimits{\mbox{minimize}}_{x \, \in \, X}
} \ \left\{\, \varphi(x) + \mathbb{E}_{\tilde\xi\sim \mathbb{P}_x}[\psi(x;\tilde\xi)] = \varphi(x) + \displaystyle{
\int_{\Xi}
} \, {\psi(x;\xi) \mbox{d}\mathbb P_x}\,\right\}.
\]
{If there exists a decision-independent distribution $\widehat{\mathbb P}$ such that $\mathbb P_x$ is absolutely continuous with respect to $\widehat{\mathbb P}$ for any $x \in X$, we can reformulate the above problem by applying \cite[Proposition 3.9]{folland1999real} as:}
\begin{equation}\label{eq:dd_transform}
    \displaystyle{\operatornamewithlimits{\mbox{minimize}}_{x \, \in \, X}} \quad \varphi(x) + \mathbb E_{\tilde\xi \sim \widehat{\mathbb P}} \,[\widehat{\psi}(x;\tilde\xi)], \;\, \mbox{with} \;\,
\widehat{\psi}(x;\tilde\xi) \, \triangleq \,  \psi(x;\tilde\xi)\displaystyle\frac{\mbox{d}\mathbb P_x}{\mbox{d}\widehat{\mathbb P}},
\end{equation}
{where $\mbox{d}\mathbb P_x/\mbox{d}\widehat{\mathbb P}$ is the Radon-Nikodym derivative of $\mathbb P_x$ with respect to $\widehat{\mathbb P}$. Even if originally  the first-stage decision $x$ only appears in the constrains of the second-stage problem in $\psi(x;\xi)$, the above transformation would make $x$ also appear in the second-stage objective function through the multiplication of $\mbox{d}\mathbb P_x/\mbox{d}\widehat{\mathbb P}$. 
}

{
A specific  example of the decision-dependent uncertainty in SPs is  a power system planning problem originated in \cite{louveaux1988optimal} and expanded in \cite{hellemo2018decision}. Assume that $\tilde\xi = (\{\tilde{d}_j\}_{j\in {\cal J}}, \{\tilde{\pi}_j\}_{j\in {\cal J}}, \{\tilde{q}_i\}_{i\in {\cal I}})$ follows a discrete distribution with the support $\{\xi^s\}_{s=1}^S$, where $\tilde{d}_j$ and $\tilde{\pi}_j$ represent the demand and the price of electricity in the location $j\in\mathcal J$, and $\tilde{q}_i$ is the unit production cost of  the power plant $i \in \mathcal I$.
For each $s = 1, \cdots, S$,
the probability of $\tilde\xi = \xi^s$ is the linear combination of  $|\mathcal G|$ given different distributions (each with probability $p_{sg}$) whose weights $\{x_g\}_{g\in\mathcal G}$ are determined as parts of 
the first-stage decisions{, i.e., $\mathbb P_x(\tilde\xi = \xi^s)=\sum_{g\in\mathcal G} p_{sg}\,x_g$ for each $s$.}
The capacity of each power plant $\{x_i\}_{i\in\mathcal I}$ also needs to be determined in the first stage. The second-stage decisions are the production $y_{ij}$ from the power plant $i$ to the  location $j$ for each  $s$.  
{By letting $\widehat{\mathbb P}(\tilde\xi = \xi^s) = 1/S$ for  each $s$ in \eqref{eq:dd_transform}}, we can rewrite the recourse function as
\begin{equation}\label{ex:recourse}
\begin{split}
	\widehat\psi(\{x_i\}_{i\in\mathcal I}, \{x_g\}_{g\in\mathcal G};\xi^s)  = &\displaystyle\operatornamewithlimits{minimum}_{l_z \leq y \leq u_z}  \;\, S \sum\limits_{g\in\mathcal G} p_{sg}\,x_g \sum\limits_{i\in\mathcal I,\, j\in\mathcal J} (q_{is}-\pi_{js})\,y_{ij}\\
    &\mbox{subject to}  \; \displaystyle\sum\limits_{j\in\mathcal J} y_{ij}\leq x_{i},\; i\in\mathcal I;
	\; \sum\limits_{i\in\mathcal I} y_{ij} = d_{js},\; j\in\mathcal J
\end{split}
\end{equation}
and obtain a decision-independent SP with the recourse $\widehat\psi$. Observe that both the objective and constraints depend on the first-stage variables. In particular, the objective function is convex in $x_g$ and concave in $y$, which fits our problem setting. Later, we will apply our proposed algorithms to solve a two-stage SP with the above recourse in Section \ref{sec:exp}.}

The second example of  the nonconvex recourse in \eqref{eq:second_stage} is the stochastic interdiction problem \cite{cormican1998stochastic,hao2020piecewise}, where the defender may want to maximize the second-stage objective function instead of minimizing it. Even for the simple linear second-stage problem with only $x$ appearing in the constraints, the recourse function 
\[
\begin{array}{rl}
    \label{eq:second_stage defend}
    \widetilde\psi(x;\xi) \, \triangleq \, 
    \displaystyle\operatornamewithlimits{maximum}_y  &\;\, c(\xi)^\top y\\[0.05in]
    \mbox{subject to}  & \;\, T(\xi)x+W(\xi)y=h(\xi)
\end{array}
\]
is not convex in $x$. One may take the dual of the second-stage maximization problem so that the recourse is a parametrized  minimization problem
\[
\begin{array}{rl}
    \label{eq:second_stage reformulate}
    \widetilde\psi(x;\xi) \, = \, 
    \displaystyle\operatornamewithlimits{minimum}_{\lambda}  &\;\, \lambda^\top T(\xi) x - \lambda^\top h(\xi)\\[0.05in]
    \mbox{subject to}  & \;\, W(\xi)^\top \lambda + c(\xi) = 0. 
\end{array}
\]
However, this dualization would bring a bilinear term $\lambda^\top T(\xi) x$ of the first-stage variable $x$ and the second-stage variable $\lambda$ to the objective function that necessities the concave-convex structure of $f(\bullet,\bullet;\xi)$. To the best of our knowledge, there is no known rigorous decomposition method to solve a general  nonconvex two-stage min-max stochastic programs  even under the linear setting.


When the distribution of $\tilde\xi$ is  taken as the empirical distribution of observed realizations $\xi^1,\ldots, \xi^S$,  the simplest way to tackle the problem \eqref{eq:first_stage} is to simultaneously solve the first-stage variable $x$ and  second-stage variables $y^1, \ldots, y^S$ (each $y^s$ is associated with one scenario $\xi^s$) via the sample average approximation  \cite{shapiro2021lectures}: 
\begin{equation}\label{determinstic}
\begin{array}{ll}
\displaystyle{
\operatornamewithlimits{\mbox{minimize}}_{x \, \in \, X, \;y^1, \ldots, y^S}
} & \;\, \varphi(x) + \displaystyle\frac{1}{S}\sum_{s=1}^S f(x,y^s;\xi^s)\\[0.15in]
    \mbox{subject to}  & \;\, G(x,y^s;\xi^s) \leq 0,\quad s=1, \ldots, S.
    \end{array}
\end{equation}
However, this approach can be prohibitive when the number of scenarios $S$ is large since the dimension of the unknown variables is $n_1+n_2 S$. Even if $S$ is small or moderate, the above formulation may still be difficult to handle under our setting as the function $f(\bullet,\bullet;\xi)$ is not jointly convex (for example when $f(\bullet,\bullet;\xi)$ is bilinear). In fact, the nonconvexity of the recourse function also makes it challenging to apply the stochastic approximation method \cite{robbins1951stochastic, polyak1990new, nemirovski2009robust} to solve \eqref{eq:first_stage}, since it is not clear how to obtain a (Clarke) subdifferential of the nonconvex  recourse function in \eqref{eq:first_stage}. Without strong assumptions like the uniqueness of the second-stage solutions, only a superset of the subdifferential $\partial \psi(\bullet;\xi)$ at given $x$ is computable \cite[Chapter 4]{bonnans2013perturbation}. When $f$ and $G$ are twice continuously differentiable, the authors in  \cite{borges2020regularized} have adopted a smoothing method to deal with the possibly nonconvex recourse  by adding the Tikhonov-regularized barrier of the inequality constraints  to the second-stage objective function. For a special class of two-stage nonconvex quadratic SPs under the simplex constraint, 
the paper \cite{bomze2022two} has derived upper and lower approximations of the objective values via copositive programs.

Notice that the constraints in \eqref{determinstic}  are in fact block-wise separable in $y^1, \ldots, y^S$ so that there is a block-angular structure between the first- and second-stage variables. 
Decomposition algorithms of two-stage SPs take advantage of this special structure to efficiently handle a large number  of scenarios via solving $S$ numbers of low-dimensional subproblems \cite{ruszczynski1997decomposition}. Two classical decomposition algorithms for two-stage SPs are  (augmented) Lagrangian decomposition and Benders decomposition.
{\it (Augmented) Lagrangian decompositions} (including the progressive hedging algorithm)  copy the first-stage variable $S$ times and attach one to each scenario \cite{guignard2003lagrangean,rockafellar1991scenarios}. In order to force the  non-anticipativity  of the first-stage decision, one has to add  equality constraints among all copies  to ensure that $x$ is the same across different realizations of the uncertainty. However, there are two major bottlenecks to apply such kind of dual-based algorithms to solve the problem \eqref{eq:first_stage}. One, each subproblem pertaining to one pair of variables $(x^s, y^s)$ is still  nonconvex if $f$ is not a jointly convex function, so that it is in general not easy to obtain its global optimal solution.  Two, the convergence of these dual approaches is largely restricted to the convex problems or special integer problems \cite{rockafellar1991scenarios,chun1995scenario}. Although there are some recent advances for the convergence study  of the progressive hedging algorithm for solving nonconvex SPs under the local convexity conditions \cite{rockafellar2019progressive,rockafellar2020augmented}, it is not clear whether the problem \eqref{eq:first_stage} satisfies those conditions without further assumptions on $f$ and $g$.
{\it Benders decomposition} (or L-shaped methods) \cite{benders1962partitioning,van1969shaped, wets1984large} alternatively updates the first-stage and second-stage variables, where
the second-stage subproblem can be solved in parallel to save the computational time and reduce storage burden. 
In order to derive 
valid inequalities of $x$ and add them to the first-stage master problem, one usually uses subgradient inequalities of the (convex) recourse function to generate a sequence of lower approximations. However, when the recourse function is associated with the complex nonconvex  function in \eqref{eq:first_stage}, it is challenging to derive its lower bounds based on the computed second-stage solutions. In fact, for the recourse functions in Figure \ref{fig:nonconvex},  there seems   not to exist a convex function that passes one of the downward cusps  and at the same time  approximates the original function from below.


In this paper, we tackle the two-stage stochastic programs \eqref{eq:first_stage} by a novel lifting technique that transforms the complex nonconvex and nondifferentiable recourse function \eqref{eq:second_stage} in the original space to a structured convex-concave function in a lifted space. The reveal of this latent structure enables us to construct convex surrogation of all recourse functions at the latest first-stage iterate,  whose evaluations are decomposable across different scenarios. Such surrogate functions are then added to the master problem to generate the next first-stage iterate. We shall prove that repeating the above procedure, the sequence of the first-stage iterates converges to a properly defined stationary solution of \eqref{eq:first_stage}.
In order to further reduce the computational cost per step when the number of scenarios $S$ is large as well as to handle the case where $\tilde\xi$ is continuously distributed, we also propose a framework that incorporates sequential sampling into the surrogation algorithm. 
The sequential sampling method gradually adds scenarios and generates cuts along the iterations, which has the advantage that one may obtain satisfactory descent progress in the early iterations with relatively small sample sizes to accelerate the overall procedure.

The paper is organized as follows. Section \ref{sec:pre} introduces notation and provides preliminary knowledge. In Section \ref{sec:approx}, we discuss the implicitly convex-concave structure of the recourse function and derive its computationally tractable approximations. A decomposition algorithm for solving problem \eqref{eq:first_stage} with a fixed number of scenarios  is proposed and analyzed in Section \ref{sec:algorithm}. To further handle the continuously distributed random vectors as well as to reduce the computational cost of the decomposition algorithm in its early stage, we provide an  internal sampling version of the algorithm in Section \ref{sec:sampling} and show the almost surely convergence of the iterative sequence.  In Section \ref{sec:exp}, we conduct extensive numerical experiments to show the effectiveness of our proposed frameworks. The paper ends with a concluding section.

\section{Preliminaries}
\label{sec:pre}
We first summarize the  notation used throughout the paper. We write $\mathbb{Z}_+$ as the set of all nonnegative integers, and $\mathbb{R}^n$ as the $n$-dimensional Euclidean space equipped with the inner product $\langle x,y\rangle = x^\top y$ and the induced norm $\|x\| \triangleq  \sqrt{x^\top x}$.
The symbol $\B(x, \delta)$ is used to denote the closed ball of radius $\delta>0$ centered at a vector $x \in \R^n$. Let $A$ and $C$ be two {nonempty subsets} of $\mathbb{R}^n$. The diameter of $A$ is defined as $R(A)\triangleq\sup\limits_{x,y\in A}\| x - y \|$, and the distance from a vector $x\in \mathbb{R}^n$ to  $A$ is defined as $\operatorname{dist}(x, A) \triangleq \inf\limits_{y \in A} \| y - x \|$. The one-sided deviation of $A$ from $C$ is defined as $\D(A, C) \, \triangleq \,  \sup\limits_{x \in A} \operatorname{dist}(x, C)$.

We next introduce the concepts of  generalized derivatives and subdifferentials for nonsmooth functions.
Interested readers are referred to the monographs \cite{clarke1990optimization, rockafellar2009variational, mordukhovich2006variational} for thorough discussions on these subjects. Consider a function $f: {\cal O} \rightarrow \R$ defined on an open set ${\cal O}\subseteq \R^n$. The classical one-sided directional derivative of $f$ at $\bar x\in {\cal O}$ along the direction $d\in\R^n$ is defined as $f^\prime(\bar x; d) \, \triangleq \, \displaystyle\lim_{t \downarrow 0} \, \frac {f(\bar x + t d) - f(\bar x)} {t}$ if this limit exists. The function $f$ is said to be directionally differentiable at $\bar x\in {\cal O}$ if it is directionally differentiable along any direction $d\in\R^n$. 
In contrast, the Clarke directional derivative of $f$ at $\bar x\in {\cal O}$ along the direction $d\in\R^n$ is defined as 
$f^\circ(\bar x; d) \,  \triangleq \, \displaystyle\operatornamewithlimits{limsup}_{x\to \bar{x},\, t\downarrow 0} \, \frac {f(x + t d) - f(x)} {t},$
which is finite when $f$ is Lipschitz continuous near $\bar x$.
        The Clarke subdifferential of $f$ at $\bar x$ is the set $\partial_C f(\bar x) \triangleq\{ v\in \mathbb{R}^n \mid f^\circ(\bar x;d) \geq v^\top d \text{ for all $d \in \R^n$} \}$, which coincides with the usual subdifferential in convex analysis for a convex function. If $f$ is strictly differentiable at $\bar{x}$, then $\partial_C f(\bar{x}) = \{\nabla f(\bar{x})\}$. 
We say that $f$ is Clarke regular at $\bar x\in {\cal O}$ if $f$ is  directionally differentiable at $\bar x$ and $f^\circ(\bar x;d)=f^\prime(\bar x;d)$ for all $d\in\R^n$. This Clarke regularity at $\bar{x}$ is equivalent to have 
$
f(x) \geq f(\bar{x}) + \bar{v}^\top (x-\bar{x}) + o(\|x-\bar{x}\|)
$ for any $\bar{v}\in \partial_C f(\bar{x})$. Therefore, if a function fails to satisfy the Clarke regularity at $\bar{x}$ (for example at the downward cusp on the left panel of Figure \ref{fig:nonconvex}), there does not exist an approximate linear lower bound of the original function based on the Clarke subdifferentials with small $o$ error locally.
    
    Let $X\subseteq \mathbb{R}^n$ be a nonempty closed convex set and $f:\mathbb{R}^n \to \mathbb{R}$ be a locally Lipschitz continuous function that is directionally differentiable. 
    We say $\bar x\in X$ is a
  directional-stationary point of $f$ on $X$ if $f^\prime (\bar x; x-\bar x) \geq 0$ for all $x \in X$, and a 
Clarke-stationary point if $f^\circ(\bar x; x-\bar x) \geq 0$ for all $x\in X$; the latter is equivalent to 
$0 \in \partial_C f(\bar x)+\mathcal N_{X}(\bar x)$ with $\mathcal N_{X}(\bar x)$ being the normal cone of $X$. 

Let $\mathcal F: \R^n \, {\rightrightarrows} \, \R^m$ be a set-valued mapping. Its outer limit at $x\in \mathbb{R}^n$ is defined as
\[
    \limsup\limits_{x \rightarrow \bar x} \mathcal F(x)
    \triangleq \bigcup\limits_{x^\nu \rightarrow \bar x} \limsup\limits_{\nu \rightarrow \infty} \mathcal F(x^\nu)
    = \big\{ u \mid \exists \; x^\nu \rightarrow \bar x, \exists \; u^\nu \rightarrow u \text{ with } u^\nu \in \mathcal F(x^\nu) \big\}.
\]
We say  $\mathcal F$ is outer semicontinuous (osc) at $\bar x\in \mathbb{R}^n$ if $\limsup\limits_{x \rightarrow \bar x} \mathcal F(x) \subseteq \mathcal F(\bar x)$.

\section{The implicit convexity-concavity  of the recourse functions}
\label{sec:approx}

A key ingredient to design a decomposition method for solving the two-stage SP \eqref{eq:first_stage} is to derive a computational-friendly approximation of the nonconvex recourse function \eqref{eq:second_stage} at any given $x\in X$ and $\xi \in \Xi$. This is the main content of the present section.

For simplicity, we omit $\xi$ in \eqref{eq:second_stage} throughout this section and rewrite the recourse function as, for $x\in  \mathbb{R}^{n_1}$,  
\begin{equation}\label{eq:value_function}
    \psi(x) \, \triangleq \, 
\left\{\begin{array}{ll}
 \displaystyle\operatornamewithlimits{minimum}_{y}  \left[\, f(x,y) \;\;
    \mbox{subject to}   \;\; G(x,y)\leq 0 \,\right] & \mbox{if {$x\in \overline X$}} \\[0.1in]
    +\infty & \mbox{if {$x\notin \overline X$}}
    \end{array}\right.,
\end{equation}
{
where $f(\bullet,\bullet)$ is concave-convex and $G(\bullet,\bullet)$ is jointly convex. We assume that for any {$x \in \overline X$}, the minimization problem of $y$ in \eqref{eq:value_function} has an optimal solution, which implies the finiteness of $\psi(x)$ on { $\overline X$}.} In the following, we show that the above function, although generally being nonconvex and nondifferentiable in $\mathbb{R}^{n_1}$, has a benign structure in a lifted space. Leveraging this structure, we then derive an approximate difference-of-convex decomposition of the recourse function that is computationally tractable.  Such an approximation is the cornerstone of  the decomposition methods to be presented in the next two sections.

\subsection{The implicit convexity-concavity of $\psi$}
As mentioned in the first section, the difficulty to design a decomposition method for solving \eqref{eq:first_stage} is due to the lack of a valid inequality of the recourse function, which is partially because 
$x$ appears in both the objective and constraints of the parametric problem in \eqref{eq:value_function}. However, if either $x$ in the objective or in the constraints is fixed, the resulting functions are relatively easy to analyze. Specifically, for any fixed ${\bar{x}\in \overline X}$, consider the functions
\[
  \psi_{\rm cvx}(x)
  \triangleq 
    \left[ \begin{array}{rl}
    \displaystyle\operatornamewithlimits{minimum}_{y} & f(\bar{x},y)\\[0.05in]
    \mbox{subject to} &   G(x,y)\leq 0
    \end{array}\hspace{-0.05in} \right]\hspace{-0.05in}\;\;\;\mbox{and}\;\;\;
\psi_{\rm cve}(x)
   \triangleq 
    \left[ \begin{array}{rl}  \displaystyle\operatornamewithlimits{minimum}_{y} &  f(x,y)\\[0.05in]
    \mbox{subject to} &  G(\bar{x},y)\leq 0
    \end{array}\hspace{-0.05in}\right].
\]
The following structural properties of $\psi_{\rm cvx}$ and $\psi_{\rm cve}$ can be easily derived. We include the proof here for completeness.

\begin{lemma}\label{lemma:icc}
    Let ${\bar x\in \overline X} \subseteq \mathbb{R}^{n_1}$ be fixed. {Suppose that the minimization problems in defining $\psi_{\rm cvx}$ and $\psi_{\rm cve}$ both have nonempty solution sets for any} {$x \in \overline X$}. Then the function $\psi_{\rm cvx}$ is convex and $\psi_{\rm cve}$ is concave on {$\overline X$}. 
\end{lemma}
\begin{proof}
   {It is not difficult to see that $\psi_{\rm cvx}$ and $\psi_{\rm cve}$ is finite for any {$x \in \overline X$} since corresponding minimization problems have nonempty solution sets.} Consider any $\lambda\in[0,1]$ and any $x^1,x^2\in {\overline X}$. We first show that
    \begin{equation}\label{eq:proof lemma}
    \psi_{\rm cvx} \big(\lambda x^1 + (1-\lambda) x^2\big) \leq \lambda \, \psi_{\rm cvx} (x^1) + (1-\lambda) \,\psi_{\rm cvx} (x^2).
    \end{equation}
    Let $\bar{y}^1$ and $\bar{y}^2$ be one of the optimal solutions of the optimization problems in defining $\psi_{\rm cvx}(x^1)$ and $\psi_{\rm cvx}(x^2)$ respectively. The joint convexity of each $g_j$ for $j=1, \cdots, \ell$ implies that
    \[
    G\big(\lambda x^1+(1-\lambda)x^2, \, \lambda \bar{y}^1+(1-\lambda)\bar{y}^2\big) \leq 0, \quad \forall \, j=1,\ldots,\ell,
    \]
which yields the feasibility of  $\lambda \bar{y}^1+(1-\lambda)\bar{y}^2$  for the minimization problem in \eqref{eq:lifting} when $x$ takes the value $\lambda x^1+(1-\lambda)x^2$. We can thus deduce that $\psi_{\rm cvx} \big(\lambda x^1 + (1-\lambda) x^2\big) \leq f \big(\bar{x}, \lambda \bar{y}^1 + (1-\lambda) \bar{y}^2\big)$. The convexity of $f(\bar{x},\bullet)$  further yields that
    \begin{gather*}
        f \big(\bar{x}, \lambda \bar{y}^1 + (1-\lambda) \bar{y}^2\big)
        \leq \lambda f (\bar{x}, \bar{y}^1) + (1-\lambda) f (\bar{x}, \bar{y}^2) = \lambda \psi_{\rm cvx} (x^1) + (1-\lambda) \psi_{\rm cvx} (x^2),
    \end{gather*}
which proves the inequality \eqref{eq:proof lemma}.
    The concavity of $\psi_{\rm cve}$ is due to the fact that it is the minimum of concave functions $\{f(\bullet,y)\}$ for $y$ feasible to the problem in \eqref{eq:lifting}.
\end{proof}

Lemma \ref{lemma:icc} suggests that the  recourse function \eqref{eq:value_function}  has a hidden convex-concave structure. Indeed, such a function $\psi$ belongs to a special class of nonconvex functions coined {\sl implicitly convex-concave} (icc) functions that are formally defined below. {For an extended-real-valued function $f:\mathbb R^n \rightarrow \mathbb R \cup \{\pm\infty\}$, the effective domain of $f$ is defined as $\operatorname{dom} f \triangleq \{x \in \mathbb R^n \mid f(x) < +\infty\}$.}
{
\begin{definition}[{\cite[Definition 4.4.4]{cui2020modern}}]\label{definition:icc}
	A function $\theta:\mathbb R^n \rightarrow \R \cup \{+\infty\}$ with $\operatorname{dom} \theta$ being a convex set is said to be implicitly convex-concave if there exists a function $\overline\theta: \mathbb R^n \times \mathbb R^n \rightarrow \R \cup \{\pm\infty\}$ satisfying:\\[0.05in]
	(a) $\overline\theta(x,z) = +\infty$ if $x \notin \operatorname{dom}\theta, z\in\mathbb R^n$, and $\overline\theta(x,z) = -\infty$ if $x\in \operatorname{dom}\theta, z\notin \operatorname{dom}\theta$;\\[0.05in]
	(b) $\overline\theta(\bullet,z)$ is convex for any fixed $z \in \operatorname{dom}\theta$;\\[0.05in]
	(c) $\overline\theta(x,\bullet)$ is concave for any fixed $x\in \operatorname{dom}\theta$;\\[0.05in]
	(d) $\theta(x) = \overline\theta(x,x)$ for any $x \in \operatorname{dom}\theta$.
\end{definition}}
The above concept is firstly introduced in  \cite{liu2020two} to analyze the convergence property of a difference-of-convex algorithm to solve two-stage convex bi-parametric quadratic SPs. More properties of icc functions are studied in the recent monograph \cite{cui2020modern}. In fact, the term ``icc'' suggests that this class of functions is a generalization of the difference-of-convex (dc) functions, as the latter is ``explicitly convex-concave'', i.e., for any dc function $\theta(x) = \theta_1(x) - \theta_2(x)$ with both $\theta_1$ and $\theta_2$ convex, one can always associate it with the bivariate function $\overline{\theta}(x,y) = \theta_1(x) - \theta_2(y)$ to explicitly expose the convexity-concavity of $\theta$ in the lifted pair $(x,y)$.
Back to the recourse function $\psi$ \eqref{eq:value_function}, we consider its lifted bivariate counterpart
\begin{equation}\label{eq:lifting}
    \overline\psi(x, z)
    \, \triangleq 
\left\{\begin{array}{ll}
 \displaystyle\operatornamewithlimits{minimum}_{y}  \left\{\, f(z,y) \mid G(x,y)\leq 0 \,\right\} & \mbox{if $x, z\in {\overline X}$} \\[0.08in]
    +\infty & \mbox{if $x\notin {\overline X}$}\\[0.08in]
    -\infty & \mbox{if $x\in {\overline X}$ and $z \notin {\overline X}$}
    \end{array}\right..
\end{equation}
 If the minimization problem of $y$ in \eqref{eq:lifting} has a nonempty solution set for any $(x,z) \in {\overline X \times \overline X}$, it is not difficult to see that the assumption in Lemma \ref{lemma:icc} holds. Henceforth, the following result is a direct consequence of Lemma \ref{lemma:icc}. No proof is needed.

\begin{proposition}\label{prop:icc}
{Assume that for any $(x,z) \in {\overline X \times \overline X}$, the minimization problem of $y$ in \eqref{eq:lifting} has a nonempty solution set.} Then $\psi$ in \eqref{eq:value_function} is an implicitly convex-concave function associated with the lifted function $\overline\psi$ in \eqref{eq:lifting}.
\end{proposition}

\begin{figure}
	\centering
    \includegraphics[scale=0.4]{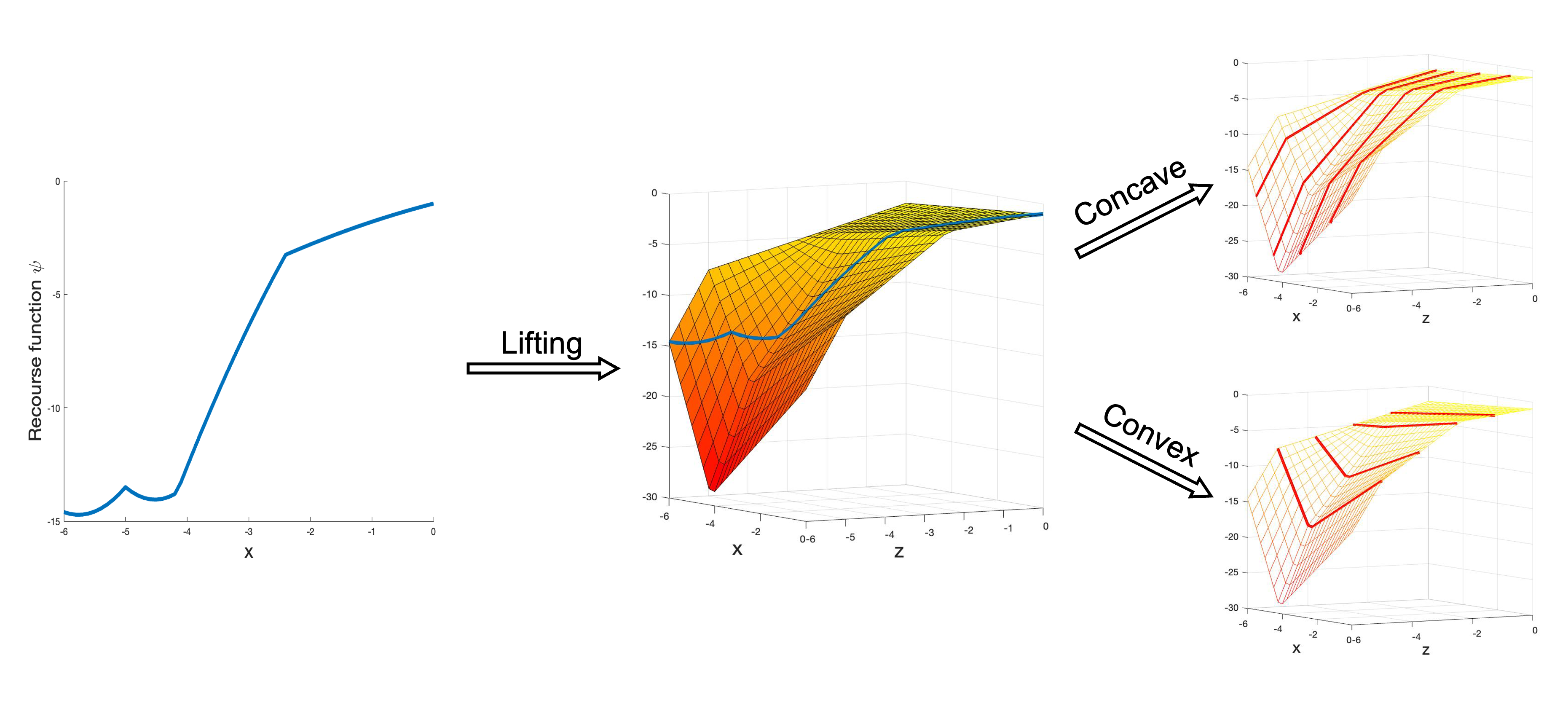}
    \centering
    \caption{\rm \scriptsize An illustration of the implicitly convex-concave structure of the nonconvex recourse function. 
    {\bf Left}: the original 1-dimensional recourse function that is neither convex nor concave; 
    {\bf Middle}: the lifted bivariate counterpart in $\mathbb{R}^2$ that is convex in $x$ and concave in $z$; {\bf Right}: the exposure of the concave component in the $z$ coordinate and  the convex component  in the $x$ coordinate.}
    \label{fig:lifting}
\end{figure}

One shall see from the subsequent sections that the derived implicitly convex-concave property of the recourse function is critical  to study the  two-stage SP \eqref{eq:first_stage}. On one hand, we can leverage this structure to construct an approximation of the nonconvex recourse function $\psi$ at any given $x$; on the other hand,  
it enables us to define  a stationary point  of \eqref{eq:first_stage}  that is  provably computable by our later designed algorithms. To fulfill these tasks, we first derive  a superset of the Clarke subdifferential of $\psi$.
To proceed, we denote $\partial_1 \overline\psi(x,z)$ as the subdifferential of the convex function $\overline\psi(\bullet,z)$ at $x$ for any $z\in {\overline X}$,  and $\partial_2 (-\overline\psi)(x,z)$ as the subdifferential of the convex function $(-\overline\psi)(x, \bullet)$ at $z$ for any $x\in {\overline X}$. We also write $\overline{Y}(x,z)$ as the set of all optimal solutions of problem \eqref{eq:lifting}.
\begin{lemma}\label{lem:subdiff}
    {Assume that for any $(x,z) \in {\overline X \times \overline X}$, the minimization problem of $y$ in  \eqref{eq:lifting} has a nonempty solution set. The following two statements hold for all $(x,z) \in {\overline X \times \overline X}$.\\[0.05in]
    (a) $\Big\{ \partial_1 (-f)(z, y) \mid y\in \overline Y(x,z) \Big\} \subseteq \partial_2(-\overline\psi)(x,z)$;}\\[0.05in]
    (b) $\partial_C \psi (x)\,\subseteq\, \partial_1 \;\overline\psi (x, x) - \partial_2 (-\overline\psi) (x, x)$.
\end{lemma}

\begin{proof}
(a) {For any $(x,z) \in {\overline X \times \overline X}$, we take any $y \in \overline Y(x,z)$ and any $c \in \partial_1 (-f)(z, y)$ to obtain
\[
    \overline\psi(x, z^\prime) \leq f(z^\prime, y)
    \leq f(z,y)+ (- c)^\top (z^\prime - z) 
    = \overline\psi(x,z) + c^\top (z - z^\prime), \quad\forall\, z^\prime \in \mathbb R^n,
\]
where the first inequality holds because $\overline{\psi}(x,z^\prime) = -\infty$ if $z^\prime \notin {\overline X}$ and $y\in\overline Y(x,z)$ must be feasible to the constraint $G(x,y) \leq 0$ if $z^\prime\in {\overline X}$; and the second inequality is due to the concavity of $f(\bullet,y)$. By applying \cite[Theorem 23.5]{rockafellar1970convex}, we have $c \in \partial_2(-\overline\psi)(x,z)$ and part (a) is proved.} Part (b) is a consequence of \cite[Proposition 4.4.26 (c)]{cui2020modern} on the relationship between the subdifferentials of an implicitly convex-concave function and its lifted counterpart.
\end{proof}

{
Part (a) of the above lemma can viewed as a weaker version of Danskin's theorem \cite{danskin2012theory, bertsekas1971control,clarke1975generalized}. Instead of a complete characterization of the subdifferential of an optimal value function in the aforementioned papers, we only need to obtain one element from this subdifferential to design our algorithms later. Therefore, only a one-sided inclusion as in part (a) is needed, which holds without the compactness of the feasible set in terms of $y$.
}

\subsection{The partial Moreau envelope}
Equipped with the lifted function $\overline\psi$, one may be able to construct  computationally friendly surrogations of the recourse function in \eqref{eq:value_function} via a modification of the usual Moreau envelope. Let us first recall the definition of the classical Moreau envelope. {An extended-real-valued function $f$ is said to be proper if $f(x) < +\infty$ for some $x\in \mathbb{R}^n$, and $f(x) > -\infty$ for all $x\in \mathbb{R}^n$.} Given a proper lower
 semicontinuous function $\theta:\mathbb{R}^n\to \R\cup\{+\infty\}$  and a positive scalar $\gamma$, its {\sl Moreau envelope}  is 
\[
    e_\gamma^{\rm ori} \,\theta(x) \, \triangleq \,  \inf\limits_{z\in \mathbb{R}^n} \left\{ \theta(z) + \frac 1 {2\gamma} \|x-z\|^2\right\}, \quad x\in \mathbb{R}^n. 
\]
We use the superscript ``ori'' to emphasize that this is the original definition of  the Moreau envelope and is different from our later modification.
The function $\theta$ is said to be prox-bounded if there exists $\gamma>0$ such that $e_\gamma^{\rm ori}\, \theta(x) > -\infty$ for some $x\in \mathbb{R}^n$. It is known that a convex function is always prox-bounded and its Moreau-envelope is continuously differentiable (c.f.\ \cite[Theorem 2.26]{rockafellar2009variational}). In general, for any prox-bounded function $\theta$, the parametric functions
$e_\gamma^{\rm ori} \,\theta(x) \uparrow \theta(x)$ as $\gamma\downarrow 0$ for all $x\in \mathbb{R}^n$. 
 Therefore, one can view the Moreau envelope as a lower approximation of the original function.
However,  if $\theta$ is nonconvex and nonsmooth, the function $e_\gamma^{\rm ori} \,\theta$ may be neither convex nor smooth. Nevertheless, for any $x\in \mathbb{R}^n$, it holds that (see, e.g., \cite{asplund1973differentiability,lucet2006fast})
\begin{equation}\label{eq:dc of moreau}
    e_\gamma^{\rm ori}\, \theta(x) \, = \frac{1}{2\gamma}\|x\|^2 -  \underbrace{\sup_{z\in \mathbb{R}^n} \left\{ -\theta(z)  - \frac{1}{2\gamma}\|z\|^2 + \frac{1}{\gamma}z^\top x \right\}}_{\text{convex in $x$ even if $\theta$ is nonconvex}},
\end{equation}
which indicates that one can always obtain a dc decomposition of $e_\gamma^{\rm ori} \,\theta$ no matter $\theta$ is convex or not. The only trouble   brought by the nonconvexity of $\theta$ is that the inner sup problem for $y$ may not be concave (especially if $\theta$ is not weakly convex), thus one may not be able to evaluate the subgradient of the second term at a given $x$ when using the dc algorithm to  minimize the function $e_\gamma^{\rm ori} \,\theta$. Specifically, in the context of the recourse function  \eqref{eq:value_function}, its associated Moreau envelope is
\[
\begin{array}{rl}
   e_\gamma^{\rm ori}\, \psi (x)
   \; = & \displaystyle\operatornamewithlimits{inf}_{z\in \mathbb{R}^n} \;\, \left\{\psi(z) + \frac 1 {2\gamma} \;\|x-z\|^2 \right\}\\[0.1in]
    = &\displaystyle\frac{1}{2\gamma}\|x\|^2 - 
    {
    \displaystyle\sup_{z\in {\overline X},y} \left\{ -f(z,y) - \displaystyle\frac{1}{2\gamma}\|z\|^2 + \frac{1}{\gamma}z^\top y \,\middle|\, G(z,y)\leq 0 \right\}
    },
\end{array}
\]
where the inner max problem is not jointly concave in $(z,y)$ since $f$ is not assumed to be jointly convex. This issue motivates us to introduce the following new type of envelopes tailored to icc functions that is more computationally tractable: 
\begin{equation}\label{eq:moreau}
    e_\gamma \theta (z)
    \triangleq \; \displaystyle\operatornamewithlimits{inf}_{x\in \mathbb{R}^n} \;\, \left\{\,\overline\theta(x,z) + \frac 1 {2\gamma} \;\|x-z\|^2 \,\right\}
\end{equation}
where $\theta:\mathbb{R}^n\to \R\cup\{+\infty\}$ is any icc function and $\overline\theta:\mathbb R^n \times \mathbb R^n \rightarrow \mathbb R \cup \{\pm\infty\}$ is its lifted counterpart  as in Definition \ref{definition:icc}. When $\overline{\theta}(x,z)$ is independent of $z$ (so that this function only has the convex part), the above definition reduces to the usual Moreau envelope. Hence, we term the new regularization of $\theta$ in \eqref{eq:moreau} its {\sl partial Moreau envelope}. Similarly as in \eqref{eq:dc of moreau}, the newly defined partial Moreau has the following explicit dc decomposition
\begin{equation}\label{eq: decompose partial}
   e_\gamma \theta (z) \,= \underbrace{\frac 1 {2\gamma} \|z\|^2}_{\text{strongly convex}} - \underbrace{\operatornamewithlimits{sup}_{x\in \mathbb{R}^n}\left\{ -\overline\theta(x, z) - \frac 1 {2\gamma}\|x\|^2 + \frac 1 \gamma z^\top x \right\}}_{\text{denoted as } g_\gamma(z),\; \text{convex}}.
\end{equation}
We denote the optimal solution mapping of the minimization problem in \eqref{eq:moreau} as 
\[
    P_{\gamma} \theta(z) \, \triangleq \, \underset{x\in \mathbb{R}^n}{\operatorname{argmin}}\left\{ \overline\theta(x,z) +\frac 1 {2\gamma} \|x-z\|^2\right\}, \quad z\in\mathbb R^n.
\]
{For any $z\in\mathbb R^n$, it holds that $\emptyset \neq P_{\gamma}\theta(z) \subseteq \operatorname{dom}\theta$}. {When $z\in\operatorname{dom}\theta$, the mapping is single-valued since the inner objective function is strongly convex in $x$; for this case, we follow the terminology in the literature to call $P_{\gamma} \theta(z)$ the proximal point of $\overline\theta$ at $z$.} 
Similar to the classical Moreau envelope, the partial Moreau envelope  approximates the original function from below. The following lemma provides the gap between the partial Moreau envelope and the original function under the 
Lipschitz continuity of $\overline\theta(\bullet,z)$. The proof is adapted from {\cite[Proposition 3.4]{planiden2019proximal}} on a similar property regarding the Moreau envelope.
\begin{lemma}\label{lem:gap_bound}
Consider an icc function $\theta: \mathbb{R}^n \to \mathbb{R}\cup\{+\infty\}$ and its lifted counterpart $\overline\theta :\mathbb{R}^n\times \mathbb{R}^n\to \mathbb{R}\cup\{\pm\infty\}$. Let $X \subseteq \operatorname{dom} \theta$. Assume that $\overline\theta(\bullet,z)$ is Lipschitz continuous {on $\operatorname{dom} \theta$} with Lipschitz constant $\kappa(z)$ for every $z\in X$, i.e.,
	\[
		\left|\overline\theta(x_1,z)-\overline\theta(x_2,z)\right| \leq \kappa(z) \, \|x_1-x_2\|, \quad \forall \, x_1, x_2\in {\operatorname{dom} \theta},\; z \in X,
	\]
then $0\leq \theta(z) - e_{\gamma}\theta(z) \;\leq\; \gamma\kappa(z)^2/2$ for any $z\in X$.
\end{lemma}
\begin{proof}
For any $z\in X$, it holds that
    \begin{align*}
    0\leq\,\theta(z) - e_{\gamma}\theta(z)
    &=\,\overline\theta(z,z) - \overline\theta\left(P_{\gamma}\theta(z),z\right) - \left\| P_{\gamma}\theta(z) - z \right\|^2/(2\gamma)\\
    &\leq\,\kappa(z) \left\|P_{\gamma}\theta(z) - z\right\| - \left\| P_{\gamma}\theta(z) - z \right\|^2/(2\gamma)
    \,\leq\,\gamma \kappa(z)^2/2,
    \end{align*}
    where the second inequality follows from the Lipschitz continuity of $\overline\theta(\bullet,z)$ and the last inequality uses the fact that $\max\limits_{t\ge 0}\left[\, \kappa(z) \,t- t^2/{(2\gamma)} \,\right] = \gamma \kappa(z)^2/2$.
\end{proof}
With $g_\gamma$ defined in \eqref{eq: decompose partial}, {it follows from similar arguments in the proof of Lemma \ref{lem:subdiff} that for any $z\in \operatorname{dom}\theta$, 
\[
   \frac 1 \gamma P_{\gamma} \theta(z)+\partial_2(-\overline\theta)(P_{\gamma} \theta(z),z) \subseteq \partial g_\gamma(z).
\]}
One can then obtain the following convex majorization of $e_\gamma\theta(z)$ at any given point $z^\prime \in \operatorname{dom}\theta$ based on the subgradient inequality of the convex function $g_\gamma$:
\begin{equation}\label{eq:upper bound general}
   e_\gamma\theta(z)\, \leq \, \widehat{e}_{\gamma}\theta(z;z^\prime)
    \triangleq \frac 1 {2\gamma} \| z\|^2 -g_\gamma(z^\prime) -\big( P_{\gamma }\theta(z^\prime)/\gamma + c\big)^\top (z - z^\prime),\;\, \forall \;z\in \operatorname{dom}\theta,
\end{equation}
where $c\in \partial_2(-\overline\theta)\left(P_{\gamma }\theta(z^\prime),z^\prime\right)$.
\section{The decomposition algorithm and its convergence}
\label{sec:algorithm} 
Based on the discussion in the last section, we are now ready to present the  decomposition algorithm for solving the nonconvex two-stage SP \eqref{eq:first_stage} and analyze its convergence.  In this section, we focus on the case where there are fixed  scenarios $\{\xi^1,\ldots,\xi^S\}$,  each realized with  probability $1/S$. The
problem \eqref{eq:first_stage} then reduces to
\begin{equation}\label{eq:finite_scenarios}
\begin{array}{rl}
   \displaystyle\operatornamewithlimits{minimize}_{x\in X\subseteq\R^{n_1}} \;\, \left\{\varphi(x)+\frac{1}{S}\sum\limits_{s=1}^S \psi(x; \xi^s)\right\},
\end{array}
\end{equation}
where each $\psi(x;\xi^s)$ is given by \eqref{eq:second_stage}. The above problem can  be viewed as a sample average approximation of the two-stage SP \eqref{eq:first_stage}
under a prescribed sample size $S$. All the discussions in this section can be easily adapted to the case where the distribution of $\tilde\xi$ has finite support (but unequal probability mass for different $\xi^s$).
We will work on the internal sampling scheme for continuously distributed $\tilde\xi$ in the next section.

\subsection{The algorithmic framework}
Our goal is to 
solve the nonconvex problem \eqref{eq:finite_scenarios} via a  successive approximation scheme. 
For any $\xi\in \Xi$ {and $z \in X$}, the partial Moreau envelope of the recourse function \eqref{eq:second_stage} associated with the bivariate function  \eqref{eq:lifting} is
\begin{equation}\label{eq:subprob}
    \begin{split}
        e_{\gamma}\psi(z;\xi)\,
        & \triangleq \; \operatornamewithlimits{minimum}_{x\in {\mathbb{R}}^{n_1}} \;\, \left\{\,\overline\psi(x, z;\xi) + \frac 1 {2\gamma} \|x-z\|^2\, \right\}\\
        &=
        \left[
        \begin{array}{ll}
		\displaystyle\operatornamewithlimits{minimum}_{x\in {\overline X}, \,  y\in \mathbb{R}^{n_2}} &\;\, f(z, y; \xi) + \displaystyle\frac 1 {2\gamma} \|x - z\|^2\\[0.1in]
        \mbox{subject to} &\;\,  G(x,y;\xi) \leq 0
        \end{array}
        \right].
    \end{split}
\end{equation}
We consider a double-loop algorithm where the outer loop updates the parameter $\gamma$ in the partial Moreau envelope and the inner loop solves the nonconvex problem $\displaystyle\operatornamewithlimits{minimize}_{x\in X}\, \left[\, \varphi(x) + \frac{1}{S} \sum_{s=1}^S e_{\gamma_\nu}\psi(x;\xi^s)\,\right]$ to stationarity for a fixed $\gamma_\nu$.
To solve the latter nonconvex problem during the $\nu$-th inner loop, we replace the Moreau-regularized recourse function $e_{\gamma_\nu}\psi(x;\xi)$ with its upper approximation constructed at the latest iterate $x_{\nu,i}$, where $i$ denotes the inner iterate index. For each $s$, let $\left(\,x^s_{\nu,i}, \, y^s_{\nu,i}\,\right)$ be one of the optimal solutions of \eqref{eq:subprob} at $(z,\xi) = (x_{\nu,i}, \xi^s)$, which can be computed by solving $s$ separable convex optimization problems. Notice that $x_{\nu,i}^s = P_{\gamma_\nu} \psi(x_{\nu,i};\xi^s) {\in \overline X}$.
One may then derive from  \eqref{eq:upper bound general} the following
upper approximating function of $e_{\gamma_\nu}\psi(x;\xi^s)$:
\begin{equation}\label{eq:Moreau_surrogate}
   \widehat{e}_{\gamma_{\nu}}\psi(x; \xi^{s}; x_{\nu,i})\triangleq \frac 1 {2\gamma_{\nu}} \|x\|^2 - g_{\gamma_{\nu}}(x_{\nu,i}; \xi^{s}) - \big( x^s_{\nu,i}/{\gamma_{\nu}} + c^s_{\nu,i}\big)^\top (x-x_{\nu,i}),
\end{equation}
where $
    g_{\gamma_{\nu}}(x_{\nu,i}; \xi^{s}) \triangleq  \displaystyle\frac 1 {2\gamma_{\nu}} \|x_{\nu,i}\|^2 - e_{\gamma_{\nu}}\psi(x_{\nu,i};\xi^{s})
$ and $c^s_{\nu,i} \in  \partial_2\left(-\overline\psi\,\right)\left(x^s_{\nu,i}, x_{\nu,i}; \xi^{s}\right)$. Due to Lemma \ref{lem:subdiff}, a particular way to choose $c^s_{\nu,i}$ is to take an element from $\partial_1 (-f)\left(x^s_{\nu,i}, y^s_{\nu,i}; \xi^{s}\right)$.
The resulting master problem to generate the next first-stage iterate $x_{\nu,i+1}$ is
\begin{equation}\label{eq:master}
    \operatornamewithlimits{minimize}_{x\in X}\;\,
    \left\{
    \varphi(x)+
    \frac{1}{S} \sum\limits_{s=1}^{S} \widehat{e}_{\gamma_{\nu}}\psi(x; \xi^{s}; x_{\nu,i})
    \right\}.
\end{equation}
The inner iteration continues until the distance between two consecutive iterates $x_{\nu,i}$ and $x_{\nu,i+1}$ is sufficiently close.
We summarize the procedure of the decomposition algorithm below. When $S=1$, it reduces to the algorithm  in \cite[Algorithm 7.2.1]{cui2020modern} to minimize an implicitly convex-concave function (without decomposition).

\begin{algorithm}[H]
\caption{A decomposition algorithm for the nonconvex two-stage SP \eqref{eq:finite_scenarios}}
\label{alg:finite_samples}
{\bf Input:} Initial point $x_0\in X$, and two scalar sequences $\{\gamma_{\nu}\}_{{\nu}\geq 0}\downarrow 0$ and $\{\varepsilon_\nu\}_{\nu\geq 0} \downarrow 0$.
\vskip 0.05in
{\bf Outer loop:} Set $\nu=0$.
\begin{algorithmic}[1]
	\State Execute the inner loop with the initial point $x_\nu$, and parameters $\gamma_\nu$ and $\varepsilon_\nu$.
	\State Set $\nu \leftarrow \nu+1$ and repeat  step 1 until a prescribed stopping criterion is satisfied.
\end{algorithmic}

\vskip 0.02in
{\bf Inner loop:} Set $i=0$ and $x_{\nu,0}=x_\nu$.
\begin{algorithmic}[1]
        \State Solve the subproblem \eqref{eq:subprob} at $(z,\xi) = (x_{\nu,i},\xi^s)$ for all $s$ to get solutions $\left(x^s_{\nu,i}, y^s_{\nu,i}\right)$.
        \State Select $c^s_{\nu,i} \in \partial_1 (-f)\left(x^s_{\nu,i}, y^s_{\nu,i}; \xi^{s}\right) \subseteq \partial_2\left(-\overline\psi\,\right)\left(x^s_{\nu,i}, x_{\nu,i}; \xi^{s}\right)$ for each $s$.
        \State Solve the master problem \eqref{eq:master} to obtain  $x_{\nu,i+1}$.
        \State Set $i \leftarrow i+1$ and repeat the above steps if $\|x_{\nu,i+1} - x_{\nu,i}\| > {\varepsilon_\nu}\gamma_\nu$. Otherwise, break the inner loop with $x_{\nu + 1} \triangleq x_{\nu,i+1}$. 
\end{algorithmic}
\end{algorithm}
Notice that each $\widehat{e}_{\gamma_{\nu}}\psi(x; \xi^{s}; x_{\nu,i})$ is a strongly convex quadratic function in $x$. Therefore, the master problem \eqref{eq:master} is a strongly convex optimization problem with $n_1$ number of variables, which is usually easy to solve.

\subsection{Convergence Analysis}
\label{sec:convergence}
This subsection is about the global convergence of the sequence generated by Algorithm \ref{alg:finite_samples}. {We begin with several technical assumptions that will be used in our convergence analysis.}
Since our focus of this section is to solve problem \eqref{eq:finite_scenarios} with fixed scenarios $\{\xi^1,\ldots,\xi^S\}$, one should interpret Assumptions A-B as requirements on all realizations $\{ \xi^1,\ldots,\xi^S \}$. The ``almost surely'' part will be used in the next section where a generally distributed $\tilde\xi$ is considered.

{
\vskip 0.05in
\noindent
\begin{center}
\fbox{%
  \parbox{0.98\textwidth}{
{\bf Assumption A.} The minimization problem of $y$ in defining $\overline\psi(x,z;\xi)$ in \eqref{eq:lifting} has an optimal solution for any {$(x,z) \in \overline X \times \overline X$} (almost surely).\\[0.05in]
{\bf Assumption B.} There exists a measurable function $\kappa_1: \Xi \rightarrow \R_+$ such that the following  condition holds (almost surely):
        \[
            \left| \overline\psi(x_1, z; \xi) - \overline\psi(x_2, z; \xi) \right| \leq  \kappa_1(\xi) \, \| x_1 - x_2 \|,\quad \forall\, x_1, x_2\in {\overline X}, z \in X.
        \]
        \vskip -0.05in
}}
\end{center}

\vskip 0.05in
Some remarks are in order.
Assumption A guarantees that for any $z\in {\overline X}$, the optimal solution of \eqref{eq:subprob} in terms of $y$ exists.
This assumption also implies the relatively complete recourse of the original problem \eqref{eq:first_stage} that $\psi(x;\xi)$ is finite for all feasible $x \in X$ almost surely.}
Assumption B is a stochastic version of the Lipschitz condition in Lemma \ref{lem:gap_bound}. Since $X$ is a compact set, we assume without loss of generality that $\kappa_1(\xi)$ is independent of $z$. 

For convenience, we denote
\begin{equation}\label{eq:zeta def}
	\overline\zeta_{S}(x) \triangleq \varphi(x)+\frac{1}{S}\sum\limits_{s=1}^{S} \psi(x; \xi^s)\quad \mbox{and} \quad \widehat\zeta_{\,S,\gamma_\nu}(x) \triangleq \varphi(x)+\frac{1}{S}\sum\limits_{s=1}^{S} e_{\gamma_\nu}\psi(x; \xi^s).
\end{equation}
In the proposition below, we  show that for any prescribed positive scalar $\varepsilon_\nu$, the $\nu$-th inner loop of Algorithm \ref{alg:finite_samples} terminates in finite steps. 

\begin{proposition}\label{thm:innerconverge1}{\rm (Convergence of the inner loop for Algorithm \ref{alg:finite_samples})}
Suppose that Assumptions A-B hold. Then the following statements hold for any $\nu$-th inner loop.\\
(a) $\widehat\zeta_{\,S,\gamma_\nu}(x_{\nu,i+1})  \leq \widehat\zeta_{\,S,\gamma_\nu}(x_{\nu,i})- \displaystyle\frac 1 {2\gamma_\nu} \|x_{\nu,i} - x_{\nu,i+1}\|^2$ for any $i\geq 0$.\\
(b) $\displaystyle\lim_{i\to 0} \|x_{\nu,i} - x_{\nu, i+1}\| = 0$ and the stopping criterion $\|x_{\nu, i+1} - x_{\nu, i}\| \leq {\varepsilon_\nu} \gamma_\nu$  is achievable in finite steps, i.e., 
\[
    i_{\nu}\triangleq\min\left\{\,i \in \mathbb{Z}_+ \left.\,\middle|\; \|x_{\nu,i+1} - x_{\nu,i}\| \leq {\varepsilon_\nu}  \gamma_\nu \right.\right\} < +\infty.
\]
In addition, we have
    \begin{align*}
        &\operatorname{dist}\left(0,
        {\begin{array}{cc}
            \displaystyle\frac 1 S\sum\limits_{s=1}^{S} \left[\partial_1 \overline\psi (x^s_{\nu,i_{\nu}}, x_{\nu, i_\nu};\xi^{s}) - \partial_2 (-\overline\psi) (x^s_{\nu,i_{\nu}},x_{\nu, i_\nu};\xi^{s})
            {+ \mathcal N_{\overline X}(x^s_{\nu,i_\nu})}\right]\\
            + \; \partial\varphi(x_{\nu, i_\nu+1})+ \mathcal N_X(x_{\nu, i_\nu + 1})
        \end{array}}
        \right) \leq \varepsilon_{\nu}.
    \end{align*}
\end{proposition}
\begin{proof}
Consider  the $\nu$-th inner loop of Algorithm \ref{alg:finite_samples}.
We first show that the sequence $\left\{\widehat\zeta_{\,S,\gamma_\nu}(x_{\nu,i})\right\}_{i\geq 0}$ is nonincreasing. Since
 the function $\widehat{e}_{\gamma_{\nu}}\psi(\,\bullet\,; \xi^{s}; x_{\nu,i})$ defined in \eqref{eq:Moreau_surrogate} is quadratic, we have, by writing $a_{\nu,i}^s \,\triangleq \, \nabla_x\,\widehat{e}_{\gamma_{\nu}}\psi(\,\bullet\,;\xi^{s}; x_{\nu,i})(x_{\nu,i+1})$,
\begin{align*}
    \widehat\zeta_{\,S,\gamma_\nu}(x_{\nu,i})
    =
    &  \;\varphi(x_{\nu,i})+
    \frac{1}{S} \sum\limits_{s=1}^{S} \widehat{e}_{\gamma_{\nu}}\psi(x_{\nu,i}; \xi^{s};  x_{\nu,i})  \\
    = &\; \varphi(x_{\nu,i}) + \frac{1}{S} \sum\limits_{s=1}^{S}\widehat{e}_{\gamma_{\nu}}\psi(x_{\nu,i+1}; \xi^{s}; x_{\nu,i})\\
    &+\frac{1}{S}\sum\limits_{s=1}^{S}{(a_{\nu,i}^s)}^\top(x_{\nu,i}-x_{\nu,i+1}) + \frac 1 {2\gamma_\nu}\|x_{\nu,i}-x_{\nu,i+1}\|^2,
 \end{align*}
 where the first equality is because  $\widehat{e}_{\gamma_{\nu}}\psi(x_{\nu,i};\xi^{s};x_{\nu,i}) = e_{\gamma_\nu}\psi(x_{\nu,i}; \xi^s)$.
Since $x_{\nu,i+1}$ is the optimal solution of the master problem \eqref{eq:master}, one may obtain that
\[
 	\left(b_{\nu,i}+\frac{1}{S}\sum\limits_{s=1}^S a^s_{\nu, i} \right)^\top(x_{\nu,i}-x_{\nu,i+1}) \geq 0 \quad\text{for some $b_{\nu, i}\in\partial\varphi(x_{\nu,i+1})$.}
 \]
The convexity of $\varphi$ and the above inequalities imply that
 \begin{equation}\label{eq:descent}
 	\begin{split}
    \widehat\zeta_{\,S,\gamma_\nu}(x_{\nu,i})\geq & \;\,\varphi(x_{\nu,i+1})+\frac{1}{S} \sum\limits_{s=1}^{S} \widehat{e}_{\gamma_{\nu}}\psi(x_{\nu,i+1};\xi^{s}; x_{\nu,i})+\frac 1 {2\gamma_\nu}\|x_{\nu,i}-x_{\nu,i+1}\|^2\\
    \geq & \;\, \varphi(x_{\nu,i+1})+\frac{1}{S} \sum\limits_{s=1}^{S} e_{\gamma_{\nu}}\psi(x_{\nu,i+1};\xi^{s})+\frac 1 {2\gamma_\nu}\|x_{\nu,i}-x_{\nu,i+1}\|^2\\
    = & \;\, \widehat\zeta_{\,S,\gamma_\nu}(x_{\nu,i+1}) + \frac 1 {2\gamma_\nu} \|x_{\nu,i} - x_{\nu,i+1}\|^2,
    \end{split}
\end{equation}
where the second inequality uses the fact that $\widehat{e}_{\gamma}\,\psi(x;\xi^s; \bar{x}) \geq e_{\gamma}\psi(x;\xi^s)$ for any $x$, $\bar{x}$ and $\xi^s$. We thus prove the part (a). Using the compactness of $X$ and Assumption B, one may further derive from Lemma \ref{lem:gap_bound} that
\[
    \inf\limits_{x \in X} \, \widehat\zeta_{\,S,\gamma_\nu}(x) \,  \geq \, \left[\, \inf\limits_{x \in X}\, \varphi(x) \,\right] + \frac 1 {S} \sum\limits_{s=1}^{S} \left[\,\inf\limits_{x \in X}\psi(x; \xi^{s}) - \frac {\gamma_{\nu}} 2 \cdot\kappa_1(\xi^s)^2 \,\right]> -\infty.
\]
Hence, the sequence $\left\{ \widehat\zeta_{\,S,\gamma_\nu}(x_{\nu,i}) \right\}_{i\geq 0}$ is bounded below, which further yields that $\left\{\widehat\zeta_{\,S,\gamma_\nu}(x_{\nu,i})\right\}_{i\geq 0}$ converges and $\|x_{\nu,i+1} - x_{\nu,i}\|$ converges to $0$ as $i\rightarrow \infty$.  The latter convergence indicates that the inner iterations terminate after finite steps.

To show the rest statement of this proposition, we first obtain from  the optimality condition of the master problem \eqref{eq:master} at $x_{\nu,i_\nu+1}$ that
\begin{equation}\label{eq:proof theorem1}
\begin{array}{rl}
    0\;\in &\partial \left(\varphi(x_{\nu,i_\nu+1})+
    \displaystyle\frac{1}{S} \sum\limits_{s=1}^{S} \widehat{e}_{\gamma_{\nu}}\psi (x_{\nu,i_\nu+1}; \xi^{s}; x_{\nu,i_\nu}) \right) + \mathcal N_X(x_{\nu,i_\nu+1})\\[0.15in]
   = & \partial\varphi(x_{\nu,i_\nu+1}) + \displaystyle\frac 1 {S}\sum\limits_{s=1}^{S} \left( \frac {x_{\nu,i_\nu+1} - x^s_{\nu,i_\nu}} {\gamma_\nu} -c^s_{\nu,i_\nu}\right) +  \mathcal N_X(x_{\nu,i_\nu+1}),
   \end{array}
\end{equation}
where the equation is due to the sum rule of the subdifferentials for convex functions \cite[Theorem 23.8]{rockafellar1970convex}.
From the optimality condition of the subproblem \eqref{eq:subprob}, we obtain
\begin{align*}
    0\;&\in\;\partial_1 \;\overline\psi\left(x^s_{\nu,i_\nu}, x_{\nu,i_\nu};\xi^{s}\right) + \left(x^s_{\nu,i_\nu} - x_{\nu,i_\nu}\right)/{\gamma_{\nu}} {+\mathcal N_{\overline X}(x^s_{\nu,i_\nu})},\quad\forall \, s=1,\ldots,S.
\end{align*}
Taking the sum over \eqref{eq:proof theorem1} and the above inclusions from $s = 1, \ldots, S$, we get 
\begin{align*}
    0
    &\in \frac 1 {S}\sum\limits_{s=1}^{S} \partial_1 \;\overline\psi\left(x^s_{\nu,i_\nu}, x_{\nu,i_\nu};\xi^{s}\right) + \left[ \partial \varphi(x_{\nu,i_{\nu+1}}) - \frac 1 S\sum\limits_{s=1}^S c^s_{\nu,i_\nu} +  \mathcal N_X(x_{\nu,i_\nu+1}) \right]\\
    &\quad + (x_{\nu,i_\nu+1} - x_{\nu,i_\nu})/\gamma_{\nu}
    {+ \frac{1}{S} \sum_{s=1}^S \mathcal N_{\overline X}(x^s_{\nu,i_\nu})}\\
    & \subseteq \frac 1 S\sum\limits_{s=1}^{S} \left[\partial_1 \overline\psi (x^s_{\nu,i_{\nu}}, x_{\nu,i_{\nu}};\xi^{s}) -  \partial_2 (-\overline\psi\,) (x^s_{\nu,i_{\nu}},x_{\nu,i_{\nu}};\xi^{s})\right]+ \partial\varphi(x_{\nu,i_\nu+1}) \\
    & \quad + (x_{\nu,i_\nu+1} - x_{\nu,i_\nu})/{\gamma_{\nu}} + \mathcal N_X(x_{\nu,i_\nu+1}) {+ \frac{1}{S} \sum_{s=1}^S \mathcal N_{\overline X}(x^s_{\nu,i_\nu})},
\end{align*}
where the last inclusion is due to  the definition of $c^s_{\nu,i_\nu}$. Consequently, we derive
\begin{align*}
    &\operatorname{dist}\left(0,
        {\begin{array}{cc}
            \displaystyle\frac 1 S\sum\limits_{s=1}^{S} \left[\partial_1 \overline\psi (x^s_{\nu,i_{\nu}}, x_{\nu,i_{\nu}};\xi^{s}) -  \partial_2 (-\overline\psi\,) (x^s_{\nu,i_{\nu}},x_{\nu,i_{\nu}};\xi^{s})
            {+ \mathcal N_{\overline X}(x^s_{\nu,i_\nu})}\right]\\
            + \; \partial\varphi(x_{\nu,i_{\nu}+1}) +  \mathcal N_X(x_{\nu,i_\nu+1})
        \end{array}}
        \right)\\
    \leq & \; \|x_{\nu,i_\nu+1} - x_{\nu,i_\nu}\|/{\gamma_{\nu}} \leq \varepsilon_{\nu},
\end{align*}
where the last inequality is due to the stopping rule of the inner loop.
\end{proof}

Let $\bar{x}$ be an accumulation point of the sequence $\{x_{\nu}\}_{\nu\geq 0}$, which must exist due to the compactness of $X$. We let $\{x_{\nu+1}\}_{\nu\in\alpha}$ be a convergent subsequence to $\bar{x}$, where $\alpha$ is a subset of $\mathbb N$. 
In the following lemma, we show the convergence of the sequence $\{x^s_{\nu,i_\nu}\}_{\nu\in \alpha}$ for each $s$ with $x^s_{\nu,i_\nu}$ being the partial optimal solution of the subproblem \eqref{eq:subprob} at $(z,\xi) = (x_{\nu, i_\nu}, \xi^s)$. 

\begin{lemma}\label{lem:prox_converge1}
Let Assumptions A-B hold and $\bar{x}$ be an accumulation point of the sequence $\{x_{\nu}\}_{\nu\geq 0}$ generated by Algorithm \ref{alg:finite_samples}. If the subsequence $\{x_{\nu+1}\}_{\nu\in\alpha}$ converges to $\bar{x}$, then $\{x_{\nu,i_\nu}\}_{\nu\in\alpha}$ converges to $\bar x$ and, for each $s=1,\ldots,S$, the corresponding subsequence  $\{x^s_{\nu,i_\nu}\}_{\nu\in\alpha}$ also converges to $\bar{x}$.
\end{lemma}
\begin{proof}
Since $\lim\limits_{\nu \rightarrow \infty} \|x_{\nu+1} - x_{\nu,i_\nu}\| = \lim\limits_{\nu \rightarrow \infty} \|x_{\nu,{i_\nu+1}} - x_{\nu,i_\nu}\| = 0$ from Theorem \ref{thm:innerconverge1}, we know that $\{x_{\nu,i_\nu}\}_{\nu\in\alpha}$ also converges to $\bar x$. For each $s=1,\ldots,S$ and $\nu\in\alpha$, we have, by the definition of $x^s_{\nu,i_\nu} = P_{\gamma_\nu}\psi(x_{\nu,i_{\nu}};\xi^s) { \, \in \overline X}$,
\begin{align*}
    \frac {1} {2 \gamma_{\nu}} \left\| x^s_{\nu,i_\nu} - x_{\nu,i_\nu}\right\|^2 + \inf_{{x \in {\overline X}}} \overline\psi (x, x_{\nu,i_\nu};\xi^{s})
    \leq \; &\frac {1} {2 \gamma_{\nu}} \left\| x^s_{\nu,i_\nu} - x_{\nu,i_\nu}\right\|^2 + \overline\psi \left(x^s_{\nu,i_\nu}, x_{\nu,i_\nu};\xi^{s}\right)\\
    = \; &e_{\gamma_{\nu}} \psi \left(x_{\nu,i_\nu}; \xi^{s}\right)
    \leq \frac {1} {2 \gamma_{\nu}} \| \bar{x} - x_{\nu,i_\nu}\|^2 + \overline\psi (\bar{x}, x_{\nu,i_\nu};\xi^{s}).
\end{align*}
The above inequality  yields that
\begin{align*}
    \left\| x^s_{\nu,i_\nu} - x_{\nu,i_\nu} \right\|
    &\leq \sqrt{\| \bar{x} - x_{\nu,i_\nu}\|^2 + 2\gamma_{\nu}\Big(\, \overline\psi (\bar{x}, x_{\nu,i_\nu};\xi^{s}) - \displaystyle\operatornamewithlimits{inf}_{{x \in {\overline X}}} \overline\psi (x, x_{\nu,i_\nu};\xi^{s}) \Big)}\\
    &\leq \| \bar{x} - x_{\nu,i_\nu} \| + \sqrt{2\gamma_\nu \Big(\, \overline\psi (\bar{x}, x_{\nu,i_\nu};\xi^{s}) - \displaystyle\operatornamewithlimits{inf}_{{x \in {\overline X}}} \overline\psi (x, x_{\nu,i_\nu};\xi^{s}) \Big)}\\
    &\leq \left\| \bar{x} - x_{\nu,i_\nu} \right\| + \sqrt{2\gamma_\nu \, \kappa_1(\xi^s) \, R({\overline X})}
    \quad \rightarrow 0 \;\,\text{as $\nu\rightarrow\infty$},
\end{align*}
where $R({\overline X})$ denotes the diameter of the compact set ${\overline X}$, and the last inequality follows from the uniform Lipschitz continuity of $\overline\psi(\bullet,z;\xi)$ on {$\overline X$} in Assumption B. Since $\kappa_1(\xi^s)<\infty$ for each $s$ and $\gamma_\nu \to 0$, we derive that
\begin{align*}
	\left\| x^s_{\nu,i_\nu} - \bar{x} \right\|
	&\leq \left\| x^s_{\nu,i_\nu} - x_{\nu,i_\nu} \right\| + \left\| x_{\nu,i_\nu} - \bar{x} \right\| \;\, \rightarrow 0 \;\,\text{as $\nu\rightarrow\infty$}\,.
\end{align*}
The proof is completed.
\end{proof}

{
For any $r>0$, we denote
\[
    \partial_1^{\,r} \overline{\psi}(\bar{x}, \bar{x}) \triangleq \bigcup\limits_{x, z \in \B(\bar{x}, r)} \partial_1 \;\overline{\psi}(x, z),
\]
and similarly for $\partial_2^{\, r}(-\overline\psi)$. We summarize a few results of the subgradient and the normal cone mappings below that will be used in the subsequent analysis.

\begin{lemma}\label{lem:subdifferential}
    Under Assumption A, the following statements hold.\\
(a) The mappings $\partial\varphi: \mathbb R^{n_1} \, {\rightrightarrows} \, \mathbb R^{n_1}$,  $\mathcal N_X: X \, {\rightrightarrows} \, \mathbb R^{n_1}$,  $\partial_1 \overline\psi: \operatorname{int}(\overline X) \times \operatorname{int}(\overline X) \, {\rightrightarrows} \, \mathbb R^{n_1}$ and $\partial_2 (-\overline\psi): \operatorname{int}(\overline X) \times \operatorname{int}(\overline X) \, {\rightrightarrows} \, \mathbb R^{n_1}$ are osc;\\
(b) For any fixed $\overline x \in X (\subseteq \operatorname{int}(\overline X))$, there exists $r > 0$ such that $\partial_1^{\,r} \overline\psi(\overline x, \overline x)$ is bounded.
\end{lemma}
\begin{proof}
    (a) The osc of $\partial\varphi$ and $\mathcal N_X$ are standard results in convex analysis (see \cite[Propositions 6.6 and 8.7]{rockafellar2009variational}). To prove that $\partial_1 \overline\psi$ is osc on $\operatorname{int}(\overline X) \times \operatorname{int}(\overline X)$, we first notice that for any fixed $d \in \mathbb R^{n_1}$, the directional derivative $\overline\psi^{\,\prime}_1 ((\bullet, \bullet); d)$ is upper semicontinuous jointly at $(\overline x,\overline z) \in \operatorname{int}(\overline X) \times \operatorname{int}(\overline X)$ \cite[Proposition 4.4.26 (a)]{cui2020modern}. Since $\overline\psi^{\,\prime}_1 ((x, z); \bullet)$ is the support function of the partial subgradient $\partial_1 \overline\psi(x, z)$, we have
    \[
        \displaystyle\operatornamewithlimits{limsup}_{(x, z) \rightarrow (\overline x, \overline z)} \left( \sup_{a \in \partial_1 \overline\psi(x, z)} a^\top d\right)
        \leq \sup_{a^\prime \in \partial_1 \overline\psi(\overline x, \overline z)} (a^\prime)^\top d \qquad \forall \,d \in \mathbb R^{n_1},
    \]
    which implies $\displaystyle\operatornamewithlimits{lim \, sup}_{(x, z) \rightarrow (\overline x, \overline z)} \partial_1 \overline\psi(x, z) \subseteq \partial_1 \overline\psi(\overline x, \overline z)$ according to \cite[Corollary 13.1.1]{rockafellar1970convex}. The osc of $\partial_2 (-\bar{\psi})$ on $\operatorname{int}(\overline X) \times \operatorname{int}(\overline X)$  can be proved similarly.\\
    (b) Suppose for the sake of contradiction that for any $r > 0$, the set $\partial_1^{\,r} \overline\psi(\overline x, \overline x)$ is unbounded. Then there exists a sequence of subgradients $c^k \in \partial_1 \overline\psi(x^k, z^k)$ with $(x^k, z^k) \rightarrow (\overline x, \overline x)$, and $\|c^k\| \rightarrow +\infty$. By taking a subsequence if necessary, we assume  that the normalized subgradient $d^k = c^k/\|c^k\|$ converges to some $d$ of unit length. Since $x^k \rightarrow \overline x \in \operatorname{int}(\overline X)$, there exists a positive scalar $t > 0$ such that $\mathbb B(\overline x, t) \subseteq \operatorname{int}(\overline X)$ and $\mathbb B(x^k, t) \subseteq \operatorname{int}(\overline X)$ for all sufficiently large $k$. Using the convexity of $\overline\psi(\bullet, z^k)$, we obtain
$
        \overline\psi(x^k + t d^k, z^k) - \overline\psi(x^k, z^k) \geq (c^k)^\top t d^k
        = t \|c^k\|$.
    Taking limits on both sides and using the continuity of $\overline\psi(\bullet, \bullet)$ on $\operatorname{int}(\overline X) \times \operatorname{int}(\overline X)$, we have
    \[
        +\infty > \overline\psi(\overline x + t d, \overline x) - \overline\psi(\overline x, \overline x) \geq t \,\lim\limits_{k \rightarrow \infty}  \|c^k\|,
    \]
    which is a contradiction. The proof is thus completed.
\end{proof}
}
We are now ready to present the global convergence of the sequence generated by Theorem \ref{thm:consistency1}. We shall prove that every accumulation point $\bar{x}$ of $\{x_{\nu}\}_{\nu\geq 0}$ is a critical point of problem \eqref{eq:finite_scenarios}  satisfying
  \begin{equation}\label{eq: critical point}
        0 \in \partial \varphi(\bar{x}) + \frac 1 S\sum\limits_{s=1}^{S} \left[\partial_1 \overline\psi\left(\bar{x}, \bar{x};\xi^{s}\right) -  \partial_2 (-\overline\psi) \left(\bar{x},\bar{x};\xi^{s}\right) \right] +  \mathcal N_{X}(\bar{x}).
    \end{equation}
It has been shown in Lemma \ref{lem:subdiff} that {under Assumption A}, we have $\partial_C \psi(x;\xi)\subseteq \partial_1 \overline\psi\left(x, x;\xi\right) - \partial_2 (-\overline\psi)(x, x;\xi)$. Hence, the  condition \eqref{eq: critical point} is weaker than the Clarke stationarity of problem \eqref{eq:finite_scenarios} pertaining to $0\in \partial \varphi(\bar{x}) + \displaystyle\frac 1 S\sum\limits_{s=1}^{S} \partial_C \psi(\bar{x};\xi^s) + \mathcal N_{X}(\bar{x})$. The term ``critical point'' is adapted from the result of the difference-of-convex (dc) algorithm to solve a dc problem $\displaystyle\operatornamewithlimits{minimize}_{x\in X}\,\left[\, \theta_1(x) - \theta_2(x)\,\right]$, where the accumulation point satisfies $0\in \partial \theta_1(x) - 
\partial\theta_2(x) + \mathcal{N}_X(x)$ \cite{an2005dc}.
\begin{theorem}[Subsequential convergence for Algorithm \ref{alg:finite_samples}]\label{thm:consistency1}
    {Let Assumptions A-B hold. Then any accumulation point of the sequence $\big\{ x_{\nu}\big\}_{\nu \geq 0}$ generated by Algorithm \ref{alg:finite_samples} is a critical point of \eqref{eq:finite_scenarios} satisfying \eqref{eq: critical point}.}
\end{theorem}
\begin{proof}
Let $\bar{x}$ be the limit of 
a convergent subsequence $\big\{x_{\nu+1}\big\}_{\nu\in \alpha}$.
By Theorem \ref{thm:innerconverge1} and Lemma \ref{lem:prox_converge1}, we know that the subsequences $\big\{x_{\nu,i_\nu}\big\}_{\nu\in \alpha}$ and $\big\{x^{s}_{\nu,i_\nu}\big\}_{\nu\in \alpha}$ for all $s$  converge to the same accumulation point $\bar{x}$ as $\nu\rightarrow\infty$. Using the triangle inequality of the distance function, we have that 
\[
\begin{array}{rl}
    &\operatorname{dist} \left( 0,
    \begin{array}{cc}
    	\displaystyle\frac 1 S\sum\limits_{s=1}^{S} \left[\partial_1 \overline\psi\left(\bar{x}, \bar{x};\xi^{s}\right) -  \partial_2 (-\overline\psi) \left(\bar{x},\bar{x};\xi^{s}\right) \right]
    	+ \; \partial \varphi(\bar{x}) + \mathcal N_{X}(\bar{x})
    \end{array}\right)\\[0.15in]
    \leq &\; \underbrace{\operatorname{dist}\left(0,
        {\begin{array}{cc}
         \; \displaystyle\frac 1 S\sum\limits_{s=1}^{S} \left[\partial_1 \overline\psi (x^s_{\nu, i_{\nu}},  x_{\nu, i_{\nu}}; \xi^{s}) -  \partial_2 (-\overline\psi) (x^s_{\nu, i_{\nu}}, x_{\nu, i_{\nu}}; \xi^{s})
         {+ \mathcal N_{\overline X}(x^s_{\nu,i_\nu})}\right] \\
       \; + \;\partial\varphi(x_{\nu,i_\nu+1}) + {w_\nu}
        \end{array}}
        \right)}_{\rm (\romannumeral1)}\\[0.15in]
 & \; + \; \displaystyle\frac 1 S\sum\limits_{s=1}^S \underbrace{\mathbb D\left(\,
    	\left[
    	\begin{array}{cc}
    		\partial_1\overline\psi(x^s_{\nu,i_\nu}, x_{\nu,i_\nu};\xi^{s} )\\[0.05in]
    		- \; \partial_2(-\overline\psi)(x^s_{\nu,i_\nu}, x_{\nu,i_\nu};\xi^{s} )
    	\end{array} \right],\;
    	\left[
    	\begin{array}{cc}
    		\partial_1\overline\psi\left(\bar{x}, \bar{x};\xi^{s}\right)\\[0.05in]
    		- \; \partial_2(-\overline\psi)\left(\bar{x}, \bar{x};\xi^{s}\right)
    	\end{array} \right] \,\right)}_{\rm (\romannumeral2)} \\[0.05in]
& \;    + \; { \displaystyle\frac{1}{S} \sum_{s=1}^S \underbrace{\mathbb D\left(\,\mathcal N_{\overline X}(x^s_{\nu,i_\nu}), \{0\}\,\right)}_{\rm (\romannumeral3)} +} 
    \underbrace{\mathbb D \left(\,\partial \varphi(x_{\nu,i_\nu+1}), \; \partial \varphi(\bar{x})\,\right)}_{\rm (\romannumeral4)}
     + \underbrace{\operatorname{dist}\left({w_\nu}, \; \mathcal N_{X}(\bar{x}) \,\right)}_{\rm (\romannumeral5)},
\end{array}
\]
{where $w_\nu$ can be any element in $\mathcal N_X(x_{\nu, i_\nu + 1})$. By Theorem \ref{thm:innerconverge1}, there is a sequence $\{w_\nu\}_{\nu \geq 0}$ with $w_\nu \in \mathcal N_X(x_{\nu, i_\nu + 1})$ such that ${\rm (\romannumeral1)}$ converges to $0$. Since for all $s$, $\lim\limits_{\nu(\in \alpha) \rightarrow \infty} x^s_{\nu, i_\nu} = \bar x \in X \subseteq \operatorname{int}(\overline X)$, we thus obtain $\mathcal N_{\overline X}(x^s_{\nu, i_\nu}) = \{0\}$ for sufficiently large $\nu$ and any $s$. Then ${\rm (\romannumeral3)}\rightarrow 0$ as $\nu(\in \alpha) \rightarrow \infty$. To see the convergence of terms ${\rm (\romannumeral2)}$, ${\rm (\romannumeral4)}$, and ${\rm (\romannumeral5)}$, based on the derived osc properties in Lemma \ref{lem:subdifferential} (a) and \cite[Proposition 5.12]{rockafellar2009variational}, we only need to prove that for each $s$, the sequences
\begin{equation}\label{eq:sequences}
    \{\partial_1 \overline\psi(x^s_{\nu, i_\nu}, x_{\nu, i_\nu}; \xi^s)\}_{\nu(\in \alpha) \geq \nu_0}, \; \{\partial \varphi(x_{\nu, i_\nu + 1})\}_{\nu(\in \alpha) \geq \nu_0}, \; \text{ and }\{w_\nu\}_{\nu(\in \alpha) \geq \nu_0}
\end{equation}
are uniformly bounded for sufficiently large $\nu_0$. Indeed, Lemma \ref{lem:subdifferential} (b) implies the existence of $r > 0$ and integer $\nu_0$ such that for any $\nu(\in \alpha) \geq \nu_0$ and any $s$,
\[
    \partial_1 \overline\psi(x^s_{\nu, i_\nu}, x_{\nu, i_\nu}; \xi^s) \subseteq \partial^{\, r}_1 \overline\psi(\overline x, \overline x; \xi^s) \text{ is bounded},
\]
which gives us the uniform boundedness of the first term in \eqref{eq:sequences}. The uniform boundedness of the second term in \eqref{eq:sequences} is a direct consequence of \cite[Theorem 24.7]{rockafellar1970convex} since $\varphi$ is real-valued and convex,  and $x_{\nu, i_\nu + 1} \rightarrow \overline x$. Observe that $\{w_\nu\}_{\nu \geq 0}$ must be bounded because ${\rm (\romannumeral1)}$ converges to $0$ and all set-valued mappings in (i) except $w_\nu$ have proven to be uniformly bounded. Henceforth, we have proved that any accumulation point $\overline x$ is a critical point of \eqref{eq:finite_scenarios} satisfying \eqref{eq: critical point}.}
\end{proof}

In the following, we establish the convergence to a stronger type of stationarity under additional assumptions. Suppose that $f(\bullet,\bullet;\xi^s)$ and $G(\bullet, \bullet; \xi^s)$ are continuously differentiable for all $\xi^s$. At $x = \bar{x}$, we say $\bar{y}^s$ is an optimal solution of the convex second-stage problem with $\xi = \xi^s$ and $\bar{\lambda}^s$ being the corresponding multiplier if the following Karush-Kuhn-Tucker condition is satisfied:
\begin{equation}\label{eq:KKT}
    0\in \nabla_y f(\bar{x},\bar{y}^{s};\xi^s)  + \; \displaystyle\sum_{j=1}^\ell\bar{\lambda}_j^s \nabla_y g_j(\bar{x},\bar{y}^s;\xi^s)\,\;\mbox{and}\;\,
    \bar{\lambda}^s \in N_{\R^\ell_-}\left(G(\bar{x},\bar{y}^s;\xi^s)\right).
\end{equation}
We use $Y(\bar{x}, \xi^s)$ and $M(\bar{x},\xi^s)$ to denote the set of all optimal solutions and multipliers satisfying the above condition, respectively. When $M(\bar{x},\xi^s)$ is nonempty, one may  write the critical cone of the second-stage problem at  $\bar{y}^s \in Y(\bar{x}, \xi^s)$ as
\begin{gather*}
	C_{\bar{x}}(\bar{y}^s,\xi^s) \triangleq \left\{ d \in \mathbb{R}^{n_2} \,\middle| \begin{array}{cc} \nabla_y f(\bar{x},\bar{y}^s;\xi^s)^\top d=0,\\[0.05in]
	\nabla_y\, g_j(\bar{x},\bar{y}^s;\xi^s)^\top d  \in \mathcal T_{\R-}(g_j(\bar{x},y;\xi)),\; j=1, \cdots, \ell \end{array} \right\}.
\end{gather*}
where ${\cal T}_D(x)$ denotes the tangent cone of a closed convex set $D$.
In the following, we show that if the second-stage solutions $\{\bar{y}^s\}_{s=1}^S$ are unique at the accumulation point $\bar{x}$ for each $s$, then $\bar{x}$ in fact satisfies a stronger condition. 
\begin{corollary}[Convergence to a directional stationary point]
\label{corollary:convergence}
	Let $\bar{x}$ be an accumulation point of the sequence $\big\{ x_{\nu}\big\}_{\nu \geq 0}$ generated by Algorithm \ref{alg:finite_samples}. In addition to the assumptions in Theorem \ref{thm:consistency1}, if the following conditions are satisfied for each $s$:
	\begin{enumerate}[leftmargin = 15pt, topsep = 2pt, itemsep = 2pt, label={(\alph*)}]
	    \item {The feasible set $\{ y \in \mathbb R^{n_2}\mid G(x,y;\xi^s) \leq 0\}$ is bounded, uniformly for $x \in X$};
		\item $f(\bullet,\bullet;\xi^s)$ and $G(\bullet,\bullet;\xi^s)$ are twice continuously differentiable;
		\item the set of multipliers $M(\bar{x}, \xi^s)$ is nonempty and 
		there exists $\bar y^s\in Y(\bar{x}, \xi^s)$  satisfying the second order sufficient condition that for all $d\in C_{\bar{x}}( \bar y^s, \xi^s)\setminus \{0\}$,
		\[
			\sup_{(\mu,\lambda)\in M(\bar y^s,\xi^s)} d^\top \nabla^2_{yy}\left[\, f(\bar{x},\bar y^s;\xi^s) + \sum_{j=1}^\ell \lambda_j g_j(\bar{x},\bar y^s;\xi^s)\,\right] d > 0;
		\]
	\end{enumerate}
	then $\bar{x}$ is  a directional stationary point of problem \eqref{eq:finite_scenarios}, i.e., 
	\[
	\varphi^\prime(\bar{x};d) + \frac{1}{S}\sum_{s=1}^S \psi^\prime(\bar{x};d) \geq 0, \quad \forall\; d\in {\cal T}_X(\bar{x}).
	\]
\end{corollary}
\begin{proof}
	We first prove that $\partial_2 (-\overline\psi)(\bar{x},\bar{x};\xi^s)$ is a singleton under given assumptions. {By condition (a) and the convexity of $(-f)(\bullet,z)$, we can apply Danskin theorem \cite[Theorem 2.1]{clarke1975generalized} to get}
	\[
		\partial_2(-\overline\psi)(\bar{x},\bar{x};\xi^s) = \operatorname{conv}\Big\{ -\nabla_x f(\bar{x},y;\xi^s) \mid y\in Y(\bar{x};\xi^s) \Big\}.
	\]
Since the second order sufficient condition of the second-stage problem holds at $\bar{x}$ for any $\xi^s$, we have that  $Y(\bar{x};\xi^s)$ is a singleton \cite[Example 13.25]{rockafellar2009variational}, which further implies that $	\partial_2(-\overline\psi)(\bar{x},\bar{x};\xi^s)$ is a singleton. The desired directional stationarity of $\bar{x}$ then follows from \cite[Proposition 6.1.11]{cui2020modern}.
\end{proof}

Theorem \ref{thm:consistency1} and Corollary \ref{corollary:convergence}
pertain to the subsequential convergence of the iterative sequence generated by Algorithm \ref{alg:finite_samples}. In the following, we show that the full sequence of the  objective values along the iterations converges if the sequence of the Moreau parameters $\{\gamma_\nu\}_{\nu\geq 0}$ is summable. This result particularly indicates that although the sequence $\{x_{\nu}\}_{\nu \geq 0}$ may have multiple accumulation points, the objective values at the accumulation points are the same.
To proceed, we remind the readers of the definition of $\overline\zeta_{S}$ in \eqref{eq:zeta def}.

\begin{theorem}[Convergence of objective values for Algorithm \ref{alg:finite_samples}]
	Suppose that Assumptions A-B hold. Let $\{ x_{\nu}\}_{\nu \geq 0}$ be the sequence generated by Algorithm \ref{alg:finite_samples} under the additional condition that $\sum\limits_{\nu=0}^\infty \gamma_\nu < \infty$. Then 
    	$\lim\limits_{\nu\rightarrow\infty}\overline\zeta_S(x_\nu) = \overline\zeta_S(\bar{x})$,
    	where $\bar{x}$ is any accumulation point of the iterative sequence $\{x_{\nu}\}_{\nu \geq 0}$.
\end{theorem}
\begin{proof}
One may derive that
	\begin{equation}\label{eq:value_diff}
		\begin{split}
		\overline\zeta_{S}(x_{\nu+1})-\overline\zeta_{S}(x_\nu)
		= \; &\Big[\, \overline\zeta_{S}(x_{\nu+1})-\widehat\zeta_{\,S,\gamma_\nu}(x_{\nu+1}) \,\Big]\\
		&+\Big[\, \widehat\zeta_{\,S,\gamma_\nu}(x_{\nu+1})-\widehat\zeta_{\,S,\gamma_\nu}(x_\nu) \,\Big]+\Big[\, \widehat\zeta_{\,S,\gamma_\nu}(x_\nu)-\overline\zeta_{S}(x_\nu) \,\Big],
		\end{split}
	\end{equation}
	where the first and last terms on the right side are gaps between the partial Moreau envelopes $\widehat\zeta_{\,S,\gamma_\nu}$ and original functions $\overline\zeta_{S}$ at $x_{\nu+1}$ and $x_\nu$, respectively. By Lemma \ref{lem:gap_bound} and {Assumption B}, we may obtain that
	\begin{equation}\label{eq:PMEerror}
		\overline\zeta_{S}(x_{\nu+1})-\widehat\zeta_{\,S,\gamma_\nu}(x_{\nu+1}) \leq \frac 1 {S} \sum\limits_{s=1}^{S} {\gamma_\nu \, {\kappa_1}(\xi^s)^2}/2 \quad\text{and}\quad \widehat\zeta_{\,S,\gamma_\nu}(x_\nu)-\overline\zeta_{S}(x_\nu) \leq 0.
	\end{equation}
	Recall that $x_{\nu+1}=x_{\nu,i_\nu+1}$ and $x_\nu=x_{\nu,0}$. Then the second term on the right side of \eqref{eq:value_diff} can be bounded above based on Theorem \ref{thm:innerconverge1}(a) that
	\begin{equation}\label{eq:Descenterror}
		\widehat\zeta_{\,S,\gamma_\nu}(x_{\nu+1})-\widehat\zeta_{\,S,\gamma_\nu}(x_\nu) \leq -\frac 1 {2\gamma_\nu} \sum\limits_{i=0}^{i_\nu}\|x_{\nu,i+1}-x_{\nu,i}\|^2.
	\end{equation}
Substituting \eqref{eq:PMEerror} and \eqref{eq:Descenterror} into the inequality \eqref{eq:value_diff}, we have
	\[
		\overline\zeta_{S}(x_{\nu+1})-\overline\zeta_{S}(x_\nu)\leq
	 - \frac 1 {2\gamma_\nu} \sum\limits_{i=0}^{i_\nu}\|x_{\nu,i+1}-x_{\nu,i}\|^2 + 	\frac {\gamma_\nu} 2 \left[\frac 1 {S} \sum\limits_{s=1}^{S}{\kappa_1}(\xi^s)^2 \right].
	\]
	Since the sequence $\{ \overline\zeta_S(x_{\nu})\}_{\nu\geq 0}$ must be bounded below due to Assumption A and $\displaystyle\sum\limits_{\nu=1}^\infty \gamma_\nu < \infty$, one may easily  obtain the convergence of $\left\{\overline\zeta_S(x_{\nu})\right\}_{\nu \geq 0}$ that is a so-called quasi-Fej{\'e}r monotone sequence; see, e.g., \cite[Lemma 3.1]{combettes2001quasi}. 
For any convergent subsequence $\{x_{\nu+1}\}_{\nu\in\alpha}$ and its
limit $\bar{x}$, we have $\lim\limits_{\nu(\in \alpha)\rightarrow\infty}\overline\zeta_S(x_{\nu+1}) = \overline\zeta_S(\bar{x})$ by  the continuity of $\psi(x;\xi^s)$ on $X$ for each $s$. Therefore, the full sequence $\left\{\overline\zeta_S(x_{\nu})\right\}_{\nu\geq 0}$ converges to $\overline\zeta_S(\bar{x})$ for any accumulation point $\bar{x}$.
\end{proof}

\section{A sampling-based decomposition algorithm}
\label{sec:sampling}
In this section, we consider a generally distributed random vector $\tilde\xi$ with a known distribution. Instead of the approach in the previous section that deals with a fixed batch of samples throughout the algorithm,
we incorporate the sampling strategy into the outer loop to progressively enlarge the problem size.
In general, there are two ways to do the sampling for solving SPs. One is to use the sample average approximation to select a subset of data before the execution of the numerical algorithm \cite{shapiro2003monte,shapiro2021lectures,homem2014monte}. The other is to adopt a sequential sampling technique \cite{higle2013stochastic,Homem-de-Mello03, pasupathy2021adaptive, RoysetPolak07} where scenarios are gradually added along the iterations. Our method falls into the latter category. 

\begin{algorithm}[H]
\caption{A sampling-based decomposition algorithm for the SP \eqref{eq:first_stage}}
\label{alg:sampling_based}
{\bf Input:} Initial point $x_0\in X$, two positive scalar sequences $\{\gamma_{\nu}\}_{{\nu}=0}^\infty \downarrow 0$, $\{\varepsilon_\nu\}_{\nu=0}^\infty \downarrow 0$, and a sequence of incremental sample size $\{S_\nu\}_{\nu=0}^\infty$.
\vskip 0.1in
{\bf Outer loop:} Set $S_{-1}=0$ and $\nu=0$.
\begin{algorithmic}[1]
	\State Generate i.i.d.\ samples $\left\{\xi^{S_{\nu-1} +\triangle}\right\}_{\triangle=1}^{S_{\nu} - S_{\nu-1}}$ from the distribution of $\tilde\xi$ that are independent of previous samples.
	\State Execute the inner loop of Algorithm \ref{alg:finite_samples} with the initial point  $x_\nu$, samples $\{\xi^s\}_{s=1}^{S_{\nu}}\triangleq\{\xi^s\}_{s=1}^{S_{\nu-1}}\bigcup\left\{\xi^{S_{\nu-1} +\triangle}\right\}_{\triangle=1}^{S_{\nu} - S_{\nu-1}}$,  and parameters $\gamma_{\nu}$ and $\varepsilon_\nu$.
	\State Set $\nu \leftarrow \nu+1$ and repeat step 1 until a prescribed stopping criterion is satisfied.
\end{algorithmic}
\end{algorithm}

We rely on the LLN for convex  subdifferentials to establish the almost surely convergence of the iterative sequence generated by Algorithm \ref{alg:sampling_based}. To facilitate this tool, Lipschitz continuity of the original function is needed. In our setting, this requires that the lifted function $\overline\psi(\bullet,\bullet;\xi)$ is Lipschitz continuous. In addition to Assumption B on the Lipschitz continuity of $\overline\psi(\bullet,z;\xi)$,  we further pose the following assumption on the Lipschitz continuity of $\overline\psi(x,\bullet;\xi)$.

\vskip 0.1in
\noindent
\begin{center}
\fbox{%
  \parbox{0.98\textwidth}{
  {\bf Assumption C.} There exists a measurable function $\kappa_2: \Xi \rightarrow \R_+$ such that $\E[\kappa_2(\xi)] < \infty$ and the following inequality holds for almost any $\xi\in \Xi$:
        \[
            \left| \overline\psi (x, x_1;\xi) - \overline\psi (x, x_2;\xi) \right| \leq  \kappa_2(\xi) \, \| x_1 - x_2 \|,\quad\forall \, x \in {\overline X}, \;x_1, x_2 \in X.
        \]
        \vskip -0.05in
}}
\end{center}
\vskip 0.1in


We have the following result on the LLN for the subdifferentials of icc functions, which is a consequence of the LLN for random set-valued mappings \cite{artstein1975strong}. In fact, the result can be viewed as a pointwise version of \cite[Theorem 2]{shapiro2007uniform}.
\begin{lemma}\label{lem:ULLN}
    Suppose that Assumptions {A-}C hold. For any {$x \in X$ and any} $r > r^\prime \geq 0$ {such that $\mathbb B(x, r) \subseteq \operatorname{int}(\overline X)$}, the following limit holds almost surely:
    \[
     \lim_{\nu\rightarrow\infty} \mathbb D \left(
        \frac 1 {S_\nu}\sum\limits_{s=1}^{S_\nu}
        \begin{bmatrix} \hskip -0.2in \partial_1^{\,r^\prime}  \,\overline\psi (x, x; \xi^{s}) \\[0.05in]
     \; - \,\partial_2^{\,r^\prime} (-\overline\psi)(x, x; \xi^{s}) \, \end{bmatrix},\hskip -0.2in
      \begin{array}{cc}\E_{\tilde\xi}\left[ \partial_1^{\,r} \, \overline\psi (x, x; \tilde\xi)\right] \\[0.05in]
        \qquad -\E_{\tilde\xi}\left[ \partial_2^{\,r} (-\overline\psi)(x, x; \tilde\xi)\right] \end{array}\right)
        = 0,
    \]
    where the expectation of a random set-valued mapping $\E_{\tilde\xi}\left[\mathcal A(x;\tilde\xi)\right]$ is defined as the set of $\E_{\tilde\xi}\left[a(x;\tilde\xi)\right]$ for all measurable selections $a(x;\xi) \in \mathcal A(x;\xi)$.
\end{lemma}
A noteworthy remark about the preceding lemma is that we can interchange the partial subdifferential and expectation in the right-hand side of the distance, i.e.,
\[
    \E_{\tilde\xi}\left[\partial_1\overline\psi(x,x;\tilde\xi)\right]=\partial_1 \, \E_{\tilde\xi}\left[\, \overline\psi(x,x;\tilde\xi)\right].
\]
This is because a convex function is Clarke regular  and the Clarke regularity ensures the interchangeability of the subdifferential and the expectation \cite[Proposition 2.3.6 and Theorem 2.7.2]{clarke1990optimization}. Next we prove a technical lemma  that is similar to Lemma \ref{lem:prox_converge1}.

\begin{lemma}\label{lem:prox_converge2}
    Suppose Assumptions A-C hold {and $\kappa_1(\xi)$ in Assumption B is essentially bounded, i.e., $\inf\left\{ t \mid \mathbb P\{\xi \in \Xi \mid |\kappa_1(\xi)| > t\} = 0\right\}$ is finite.} If a subsequence $\{x_{\nu+1}\}_{\nu\in \alpha}$ generated by Algorithm \ref{alg:sampling_based} converges to a point $\bar{x}$, then the subsequence $\{x_{\nu, i_\nu}\}_{\nu\in \alpha}$ also converges to $\bar x$, and for any $r>0$, the subsequence $\left\{x^s_{\nu,i_\nu}\right\}_{s=1}^{S_\nu} \subset \B(\bar{x},r)$ almost surely for all $\nu\in\alpha$ sufficiently large.
\end{lemma}
\begin{proof}
	Following similar derivation in the proof of  Lemma \ref{lem:prox_converge1}, we obtain the convergence of $\{x_{\nu, i_\nu}\}_{\nu\in \alpha}$ to $\bar x$ and the following inequalities
	\begin{align*}
	\left\| x^s_{\nu,i_\nu} - \bar{x} \right\|
	&\leq \left\| x^s_{\nu,i_\nu} - x_{\nu,i_\nu} \right\| + \left\| x_{\nu,i_\nu} - \bar{x} \right\|\\
	&\leq 2 \left\| x_{\nu,i_\nu} - \bar{x} \right\| + \sqrt{2\gamma_\nu \Big(\, \overline\psi (\bar{x}, x_{\nu,i_\nu};\xi^{s}) - \inf\limits_{x \in {\overline X}} \overline\psi (x, x_{\nu,i_\nu};\xi^{s}) \Big)}\\
	&\leq 2 \left\| x_{\nu,i_\nu} - \bar{x} \right\| + \sqrt{2\gamma_\nu \, \kappa_1(\xi^s) \, R({\overline X})}.
\end{align*}
{Notice that $\gamma_\nu \to 0$ and $\kappa_1(\xi^s)$ is almost surely bounded by a constant independent of $s$. Thus, for any given $r > 0$, there is a positive integer $N$ such that $\|x^s_{\nu, i_\nu} - \bar x\| \leq r$ holds for any $s = 1, 2,\cdots, S_\nu$ and all $\nu (\in \alpha) \geq N$.}
\end{proof}

Below is the main theorem of this section on the almost surely subsequential convergence  of the iterative sequence generated by Algorithm \ref{alg:sampling_based}.

\begin{theorem}[Subsequential convergence of Algorithm \ref{alg:sampling_based}]\label{thm:consistency2}
    Suppose that Assumptions A-C hold {and $\kappa_1(\xi)$ in Assumption B is essentially bounded.} Let $\{ x_{\nu}\}_{\nu \geq 0}$ be the sequence generated by Algorithm \ref{alg:sampling_based} and $\bar{x}$ be any accumulation point. For any $r>0$ {such that $\mathbb B(\bar x, r) \subseteq \operatorname{int}(\overline X)$}, the following inclusion holds almost surely: 
    	\[
        	0 \in \partial \varphi(\bar{x}) + \partial^{\,r}_1 \, \E_{\tilde\xi}\left[\, \overline\psi(\bar{x}, \bar{x}; \tilde\xi) \right] - \partial^{\,r}_2 \,\E_{\tilde\xi}\left[ (-\overline\psi)(\bar{x}, \bar{x}; \tilde\xi)\right] + \mathcal N_{X}(\bar{x}).
    	\]
In addition, if the set-valued mapping $\partial_1 \, \E_{\tilde\xi}\left[\, \overline\psi(\bullet, \bullet; \tilde\xi)\right] - \partial_2 \, \E_{\tilde\xi}\left[\, (-\overline\psi)(\bullet, \bullet; \tilde\xi) \right]$ is continuous at $(\bar{x}, \bar{x})$, then almost surely 
    \[
        0 \in \partial \varphi(\bar{x}) + \partial_1 \, \E_{\tilde\xi}\left[\, \overline\psi(\bar{x}, \bar{x}; \tilde\xi) \right] - \partial_2\,\E_{\tilde\xi}\left[ (-\overline\psi)(\bar{x}, \bar{x}; \tilde\xi) \right] + \mathcal N_{X}(\bar{x}),
    \]
i.e., every accumulation point $\bar{x}$ is a  critical point of problem \eqref{eq:first_stage} almost surely.
\end{theorem}

\begin{proof}
Consider any subsequence $\{x_{\nu+1}\}_{\nu\in \alpha}$ that converges to $\bar{x}$.  It follows from Lemma \ref{lem:prox_converge2} that for any $r^\prime>0$, the ball $\B(\bar{x},r^\prime)$ almost surely contains $x_{\nu,i_\nu}, x_{\nu,i_\nu+1},$ and the proximal points $\{x^s_{\nu,i_\nu}\}_{s=1}^{S_\nu}$ for all $\nu(\in\alpha)$ sufficiently large. 
Consequently, 
\begin{gather*}
    \partial_1 \,\overline\psi ( x^{s}_{\nu,i_\nu}, x_{\nu,i_\nu}; \xi^{s})
    \subseteq
    \partial_1^{\,r^\prime} \, \overline\psi (\bar{x}, \bar{x}; \xi^{s}) \quad \text{almost surely for all $\nu\in\alpha$ sufficiently large.}
\end{gather*}
We then obtain from Lemma \ref{lem:ULLN} that, {for any $r>0$ such that $\mathbb B(\bar x, r) \subseteq \operatorname{int}(\overline X)$}, the following limit holds almost surely:
\begin{equation}\label{eq:ULLN2}
    \lim_{\nu\rightarrow \infty} \mathbb D \left(
        \frac 1 {S_\nu}\sum\limits_{s=1}^{S_\nu} \begin{bmatrix}\hskip -0.3in \partial_1 \,\overline\psi (x^s_{\nu,i_\nu}, x_{\nu,i_\nu}; \xi^{s}) \\[0.05in] 
        \;- \; \partial_2 (-\overline\psi)(x^s_{\nu,i_\nu}, x_{\nu,i_\nu}; \xi^{s})\end{bmatrix}
        ,\begin{array}{lc} \partial^{\,r}_1 \,\E_{\tilde\xi}\left[\, \overline\psi(\bar{x}, \bar{x}; \tilde\xi)\right] \\[0.05in] - \; \partial^{\,r}_2 \,\E_{\tilde\xi}\Big[ (-\overline\psi)(\bar{x}, \bar{x}; \tilde\xi)\Big]\end{array} \right)
        = 0.
\end{equation}
Thus, the following estimation follows almost surely:
\begin{align*}
    &\operatorname{dist} \left( 0, \; \partial \varphi(\bar{x}) + \partial^{\,r}_1 \, \E_{\tilde\xi}\left[\, \overline\psi(\bar{x}, \bar{x}; \tilde\xi)\right] - \partial^{\,r}_2 \,\E_{\tilde\xi}\left[ (-\overline\psi)(\bar{x}, \bar{x}; \tilde\xi)\right] + \mathcal N_{X}(\bar{x}) \right)\\
    &\leq
    \underbrace{\operatorname{dist}\left(0,
        {\begin{array}{cc}
             \displaystyle\frac 1 {S_\nu}\sum\limits_{s=1}^{S_\nu}\left[\partial_1 \, \overline\psi(x^s_{\nu,i_\nu}, x_{\nu,i_\nu}; \xi^{s}) -  \partial_2 (-\overline\psi)(x^s_{\nu,i_\nu}, x_{\nu,i_\nu}; \xi^{s})
             {+ \mathcal N_{\overline X}(x^s_{\nu,i_\nu})}\right]\\
            + \; \partial\varphi(x_{\nu,i_\nu+1}) + {w_\nu}
        \end{array}}
        \right)}_{\rm (\romannumeral1^\prime)}\\[0.05in]
     & + \underbrace{\mathbb D \left(
        \displaystyle\frac 1 {S_\nu}\sum\limits_{s=1}^{S_\nu} \begin{bmatrix}\hskip -0.3in \partial_1 \,\overline\psi (x^s_{\nu,i_\nu}, x_{\nu,i_\nu}; \xi^{s}) \\[0.05in]
        \; - \; \partial_2 (-\overline\psi)(x^s_{\nu,i_\nu}, x_{\nu,i_\nu}; \xi^{s})\end{bmatrix}
        ,\begin{array}{lc} \partial^{\,r}_1 \,\E_{\tilde\xi}\left[\, \overline\psi(\bar{x}, \bar{x}; \tilde\xi)\right] \\ - \; \partial^{\,r}_2 \,\E_{\tilde\xi}\Big[ (-\overline\psi)(\bar{x}, \bar{x}; \tilde\xi)\Big]\end{array} \right)}_{\rm (\romannumeral2^\prime)}\\[0.05in]
     & { + \displaystyle \underbrace{\frac{1}{S_\nu} \sum_{s=1}^{S_\nu} \mathbb D\left(\,\mathcal N_{\overline X}(x^s_{\nu,i_\nu}), \{0\}\,\right)}_{\rm (\romannumeral3^\prime)}}
     + \underbrace{\mathbb D \left(\,\partial \varphi(x_{\nu,i_\nu+1}), \partial \varphi(\bar{x})\,\right)}_{\rm (\romannumeral4^\prime)}
     + \underbrace{\operatorname{dist}\left( {w_\nu}, \, \mathcal N_{X}(\bar{x}) \right)}_{\rm (\romannumeral5^\prime)},
\end{align*}
{
where $w_\nu$ can be any element in $\mathcal N_X(x_{\nu, i_\nu + 1})$. By Theorem \ref{thm:innerconverge1} (with sample size $S_\nu$ instead of $S$), there is a sequence $\{w_\nu\}_{\nu \geq 0}$ with $w_\nu \in \mathcal N_X(x_{\nu, i_\nu + 1})$ such that ${\rm (\romannumeral1^\prime)}$ converges to $0$. As shown in \eqref{eq:ULLN2}, we have ${\rm (\romannumeral2^\prime)} \rightarrow 0$. The term ${\rm (\romannumeral3^\prime)} \rightarrow 0$ because $\{x^s_{\nu,i_\nu}\}_{s=1}^{S_\nu} \subseteq \mathbb B(\bar x, r) \subseteq \operatorname{int}(\overline X)$ holds almost surely for sufficiently large $\nu$. The convergence of the last two terms ${\rm (\romannumeral4^\prime)}$ and ${\rm (\romannumeral5^\prime)}$ to $0$ can be derived based on similar arguments to their counterparts in the proof of Theorem \ref{thm:consistency1}.
}
Finally, if the set-valued mapping $\partial_1 \, \E_{\tilde\xi}\left[\, \overline\psi(\bullet, \bullet;\tilde\xi)\right] - \partial_2\,\E_{\tilde\xi}\left[(-\overline\psi)(\bullet, \bullet;\tilde\xi)\right]$ is continuous at $(\bar{x}, \bar{x})$, one may adopt similar arguments as the proof of \cite[Theorem 3]{shapiro2007uniform} to derive the almost surely convergence to a  critical point. 
\end{proof}


One can further derive an analogous result of Corollary \ref{corollary:convergence} for the sequence generated by  Algorithm \ref{alg:sampling_based} by strengthening the conditions (a), (b), and (c) in the former corollary to almost any $\xi\in \Xi$ so that $\partial_2 \E_{\tilde\xi}\left[(-\overline\psi)(\bar{x}, \bar{x};\tilde\xi)\right]$ is a singleton. We omit the details here for brevity. 
The last result of this section is the almost surely convergence of the objective values of 	$\{\overline\zeta_{S_\nu}(x_\nu)\}_{\nu\geq 0}$ under proper assumptions on the sample sizes $S_\nu$ and Moreau parameters $\gamma_\nu$. To proceed, we first
present a lemma on the convergence rate of the SAA in expectation. This result is obtained by using the Rademacher average of the random function $\psi(x;\tilde\xi)$, which has its source in \cite[Corollary 3.2]{ermoliev2013sample}; see also \cite[Theorem 10.1.5]{cui2020modern}.

\begin{lemma}\label{lem:ULLNrate}
Let $X$ be a compact set in $\R^n$ and $\tilde\xi: \Omega \rightarrow \Xi\subseteq\R^{m}$ be a random vector defined on a probability space $(\Omega, \mathcal F, \mathbb P)$. Let $\psi: X\times\Xi \rightarrow \R$ be a Carath{\'e}odory function.
Suppose that $\psi$   is uniformly bounded on $X\times\Xi$ and
          Lipschitz continuous in $x$ with modulus independent of $\xi$.
    Let $\{\xi^s\}_{s=1}^S$ be independent and identically distributed  random vectors following the distribution of $\tilde\xi$. Then there exists a constant $C$ such that for any $\eta\in(0,1/2)$, we have
    \[
    	\E\left[\, \sup\limits_{x\in X}\left|\frac 1 {S}\sum\limits_{s=1}^{S}\psi(x;\xi^s) - \E_{\tilde\xi}\left[ \psi(x;\tilde\xi)\right]\right| \,\right] \leq \frac {C \sqrt{1-2\eta}}{S^\eta}, \quad \forall\, S>0.
    \]
\end{lemma}

We make a remark about the above lemma. For an icc function $\psi(x;\xi)$ associated with the lifted counterpart $\overline\psi(x,z;\xi)$, the uniform Lipschitz continuity of $\psi(\bullet;\xi)$ holds on $X$ when $\overline\psi(\bullet,z;\xi)$ is uniformly Lipschitz continuous over $(z;\xi) \in X\times\Xi$ and $\overline\psi(x,\bullet;\xi)$ is uniformly Lipschitz continuous over $(x;\xi) \in X\times\Xi$. Indeed, one can deduce this uniform Lipschitz continuity from Assumptions B and C with $\displaystyle\sup_{\xi\in \Xi} \left[\,\max(\kappa_1(\xi), \, \kappa_2(\xi) )\,\right]< \infty${, which also implies the essential boundedness of $\kappa_1(\xi)$ assumed in Lemma \ref{lem:prox_converge2} and Theorem \ref{thm:consistency2}}. We are now ready to present the almost surely sequential convergence
of the objective values generated by the internal sampling scheme.

\begin{theorem}[Sequential convergence of objective values for Algorithm \ref{alg:sampling_based}]
	Suppose that Assumptions A-C and conditions for the function $\psi$ in Lemma \ref{lem:ULLNrate} hold. Let $\{ x_{\nu}\}_{\nu \geq 0}$ be the sequence generated by Algorithm \ref{alg:sampling_based}. Assume that the  parameter of the partial Moreau envelope $\gamma_\nu$ and the sample size $S_\nu$ satisfy
        \[
        	\sum\limits_{\nu=1}^\infty \gamma_\nu < \infty, \qquad \sum\limits_{\nu=1}^\infty \frac {S_{\nu+1}-S_\nu} {S_{\nu+1} \, (S_\nu)^\eta}  < \infty \quad \text{for some }\eta\in(0,1/2).
        	\]
	Then 
    	$\lim\limits_{\nu\rightarrow\infty}\overline\zeta_{S_\nu}(x_\nu) = \overline\zeta(\bar{x})$ almost surely,
    	where $\bar{x}$ is any accumulation point of the iterative sequence $\{x_{\nu}\}_{\nu \geq 0}$.
\end{theorem}
\begin{proof}
	We first prove the almost sure convergence of $\left\{\overline\zeta_{S_\nu}(x_\nu)\right\}_{\nu \geq 0}$. We have
	\begin{equation}\label{estimate2}
		\begin{array}{ll}
		\overline\zeta_{S_{\nu+1}}(x_{\nu+1})-\overline\zeta_{S_\nu}(x_\nu)
		= &
		\underbrace{\left[\, \overline\zeta_{S_{\nu+1}}(x_{\nu+1})-\overline\zeta_{S_\nu}(x_{\nu+1}) \right]}_{\triangleq \, R_{\nu,1}} +\underbrace{\Big[\, \overline\zeta_{S_\nu}(x_{\nu+1})-\widehat\zeta_{\,S_\nu,\gamma_\nu}(x_{\nu+1}) \Big]}_{\triangleq \, R_{\nu,2}}\\
		&+\underbrace{\Big[\, \widehat\zeta_{\,S_\nu,\gamma_\nu}(x_{\nu+1})-\widehat\zeta_{\,S_\nu,\gamma_\nu}(x_\nu) \Big]}_{\triangleq \, R_{\nu,3}}+\underbrace{\Big[\, \widehat\zeta_{\,S_\nu,\gamma_\nu}(x_\nu)-\overline\zeta_{S_\nu}(x_\nu) \Big]}_{\triangleq \, R_{\nu,4}}.
		\end{array}
	\end{equation}
	Using results of \eqref{eq:PMEerror} and \eqref{eq:Descenterror}, we obtain
	\begin{align*}
		R_{\nu,2} \leq \frac 1 {S_\nu} \sum\limits_{s=1}^{S_\nu} \frac {\gamma_\nu \, {\kappa_1}(\xi^s)^2} {2},\quad
		R_{\nu,3} \leq -\frac 1 {2\gamma_\nu} \sum\limits_{i = 0}^{i_\nu}\|x_{\nu,i + 1}-x_{\nu,i}\|^2\quad \mbox{and}\quad
		R_{\nu,4} \leq 0.
	\end{align*}
	Next we  compute $R_{\nu,1}$ that is the error of the sample-augmentation. It holds that
	\begin{align*}
		R_{\nu,1}
		&=\frac 1 {S_{\nu+1}}\left[\sum\limits_{s=1}^{S_\nu} \psi(x_{\nu+1};\xi^s)+\sum\limits_{s=S_\nu+1}^{S_{\nu+1}} \psi(x_{\nu+1};\xi^s)\right] - \frac 1 {S_\nu}\sum\limits_{s=1}^{S_\nu} \psi(x_{\nu+1};\xi^s)\\
		&=\left(\frac {S_\nu} {S_{\nu+1}}-1\right) \frac 1 {S_\nu}\sum\limits_{s=1}^{S_\nu} \psi(x_{\nu+1};\xi^s) + \frac 1 {S_{\nu+1}} \sum\limits_{s=S_\nu+1}^{S_{\nu+1}} \psi(x_{\nu+1};\xi^s)\\
		&=\left(\frac {S_\nu} {S_{\nu+1}}-1\right) \left[\frac 1 {S_\nu}\sum\limits_{s=1}^{S_\nu} \psi(x_{\nu+1};\xi^s)-\frac 1 {S_{\nu+1}-S_\nu}\sum\limits_{s=S_\nu+1}^{S_{\nu+1}} \psi(x_{\nu+1};\xi^s)\right].
	\end{align*}
	Let $\mathcal F_\nu\triangleq\sigma(\xi^1,\xi^2,\ldots,\xi^{S_\nu})$ be a filtration, i.e., an increasing sequence of $\sigma-$fields generated by samples used in outer iterations. Obviously $x_{\nu+1}$ is adapted to $\mathcal F_\nu$ and $\{\xi^s\}_{s=S_\nu+1}^{S_{\nu+1}}$ are independent of $\mathcal F_\nu$. Therefore, by taking conditional expectation of $R_{\nu,1}$ given $\mathcal F_\nu$, we obtain
	\[
		\E[\,R_{\nu,1}\mid \mathcal F_\nu\,]=\left(\frac {S_\nu} {S_{\nu+1}}-1\right) \left[\frac 1 {S_\nu}\sum\limits_{s=1}^{S_\nu} \psi(x_{\nu+1};\xi^s)-\E_{\tilde\xi}\left[\, \psi(x_{\nu+1};\tilde\xi) \right]\right],
	\]
	where $\{\xi^s\}_{s=1}^{S_\nu}$ and $\tilde\xi$ are independent and identically distributed. Based on the estimations of the terms $R_{\nu,1},R_{\nu,2},R_{\nu,3}$ and $R_{\nu,4}$, we have, by taking conditional expectation of \eqref{estimate2} given $\mathcal F_\nu$,
	\begin{equation}\label{eq:superMG}
		\begin{split}
		&\E\Big[\, \overline\zeta_{S_{\nu+1}}(x_{\nu+1})\mid\mathcal F_\nu\Big] - \overline\zeta_{S_\nu}(x_\nu) - \frac 1 {2\gamma_\nu} \sum\limits_{i = 0}^{i_\nu}\|x_{\nu,i + 1}-x_{\nu,i}\|^2\\
		\leq &\; \frac {S_{\nu+1}-S_\nu} {S_{\nu+1}} \left|\frac 1 {S_\nu}\sum\limits_{s=1}^{S_\nu} \psi(x_{\nu+1};\xi^s)-\E_{\tilde\xi}\left[\, \psi(x_{\nu+1};\tilde\xi)\right]\right| + \frac 1 {S_\nu} \sum\limits_{s=1}^{S_\nu} \frac {\gamma_\nu \, {\kappa_1}(\xi^s)^2} 2.
		\end{split}
	\end{equation}
	In order to show the almost sure convergence of $\left\{ \overline\zeta_{s_\nu}(x_\nu) \right\}_{\nu \geq 0}$, we need to verify that the right side of the preceding inequality is summable over $\nu$ almost surely and the sequence $\left\{\overline\zeta_{S_\nu}(x_\nu)\right\}_{\nu \geq 0}$ is bounded below almost surely. We have
	\begin{align*}
		&\E \left[\, \sum\limits_{\nu=1}^\infty \left(\frac {S_{\nu+1}-S_\nu} {S_{\nu+1}}\right) \left|\frac 1 {S_\nu}\sum\limits_{s=1}^{S_\nu} \psi(x_{\nu+1};\xi^s)-\E_{\tilde\xi}\left[\psi(x_{\nu+1};\tilde\xi)\right]\right| \,\right]\\
		= \;&\sum\limits_{\nu=1}^\infty \frac {S_{\nu+1}-S_\nu} {S_{\nu+1}} \, \E \left[\, \left|\frac 1 {S_\nu}\sum\limits_{s=1}^{S_\nu} \psi(x_{\nu+1};\xi^s)-\E_{\tilde\xi}\left[\, \psi(x_{\nu+1};\tilde\xi)\right]\right| \,\right]\\
		\leq \;& \sum\limits_{\nu=1}^\infty \frac {S_{\nu+1}-S_\nu} {S_{\nu+1}} \; \E \left[\, \sup_{x\in X}\left|\frac 1 {S_\nu}\sum\limits_{s=1}^{S_\nu} \psi(x;\xi^s)-\E_{\tilde\xi}\left[\, \psi(x;\tilde\xi)\right]\right| \,\right]\\
		\leq \;&\sum\limits_{\nu=1}^\infty \frac {(S_{\nu+1}-S_\nu) \, C \sqrt{1-2\eta}} {S_{\nu+1} \, (S_\nu)^{\eta}} < \infty, \; \text{for some $\eta\in(0,1/2)$, by Lemma \ref{lem:ULLNrate}}.
	\end{align*}
	Hence, we derive that
	\[
		\sum\limits_{\nu=1}^\infty \left(\frac {S_{\nu+1}-S_\nu} {S_{\nu+1}}\right) \left|\frac 1 {S_\nu}\sum\limits_{s=1}^{S_\nu} \psi(x_{\nu+1};\xi^s)-\E_{\tilde\xi}\left[\psi(x_{\nu+1};\tilde\xi)\right]\right| < \infty \quad \text{almost surely.}
	\]
	Since $\{\gamma_\nu\}_{\nu=0}^\infty$ is assumed to be summable, we can obtain
	\[
		\E\left[\, \sum\limits_{\nu=0}^\infty \left(\frac 1 {S_\nu} \sum\limits_{s=1}^{S_\nu} \frac {\gamma_\nu \, {\kappa_1}(\xi^s)^2} 2\right) \,\right] = \sum\limits_{\nu=0}^\infty\E \left[\frac {\gamma_\nu} {S_\nu} \sum\limits_{s=1}^{S_\nu} \frac {{\kappa_1}(\xi^s)^2} 2\right] = \sum\limits_{\nu=0}^\infty \frac {\gamma_\nu \, \E_{\tilde\xi}\left[{\kappa_1}(\tilde\xi)^2\right]} 2 < \infty.
	\]
	Consequently, $\displaystyle\sum\limits_{\nu=1}^\infty \left(\frac 1 {S_\nu} \sum\limits_{s=1}^{S_\nu} \frac {\gamma_\nu \, \kappa(\xi^s)^2} 2 \right) < \infty$ almost surely. We have thus proved that the right side of \eqref{eq:superMG} is summable over $\nu$ almost surely. Next, we show that $\left\{\overline\zeta_{S_\nu} (x_\nu)\right\}_{\nu \geq 0}$ is bounded below almost surely. To see this, note that
	\[
		\sup_{x\in X}\left|\overline\zeta_{S_\nu} (x)\right| \leq \sup_{x\in X}\left| \overline\zeta_{S_\nu} (x) - \zeta(x)\right| + \sup_{x\in X} \left|\zeta(x)\right|,
	\]
	where the first term converges to $0$ almost surely by the uniform LLN (c.f. \cite[Theorem 9.60]{shapiro2021lectures}) and the second one is bounded due to the continuity of $\zeta(x) = \varphi(x) + \mathbb{E}_{\xi}\left[ \, \psi(x;\xi) \, \right]$ on the compact set $X$. Therefore, there exists a constant $M$ such that $\overline\zeta_{S_\nu}(x_\nu)$ is bounded below by $M$ almost surely for any $\nu$. Applying Robbins-Siegmund nonnegative almost supermartingale convergence lemma (c.f.\ \cite[Theorem 1]{robbins1971convergence}), we have $\displaystyle\sum\limits_{\nu=1}^\infty \displaystyle\frac 1 {2\gamma_\nu} \sum\limits_{i = 0}^{i_\nu}\|x_{\nu,i + 1}-x_{\nu,i}\|^2 < \infty$ almost surely and the sequence $\left\{\overline\zeta_{S_\nu}(x_\nu)\right\}_{\nu \geq 0}$ converges almost surely. Finally,
	let $\bar{x}$ be the limit of a convergent subsequence $\{x_{\nu}\}_{\nu\in\alpha}$. Using the uniform convergence of $\overline\zeta_{S_\nu}$ to $\zeta$ and the continuity of $\zeta$ on the compact set $X$, it follows from \cite[Proposition 5.1]{shapiro2021lectures} that $\overline\zeta_{S_{\nu}}(x_{\nu})$ converges to $\zeta(\bar{x})$ almost surely as $\nu(\in\alpha)\rightarrow\infty$. This argument, together with the convergence of the full sequence $\left\{ \overline\zeta_{S_\nu}(x_\nu) \right\}_{\nu \geq 0}$, completes the proof of this theorem. 
\end{proof}

\section{Numerical  experiments}
\label{sec:exp}

In this section, {we present numerical results for a power system planning problem with the recourse in \eqref{ex:recourse}  and a linear first-stage objective of $x = \left(\{x_i\}_{i\in {\cal I}}, \{x_g\}_{g\in {\cal G}}\right)$.} The overall deterministic equivalent  formulation to minimize the total cost  is
\begin{equation}\label{eq:testprob}
\begin{split}
  \displaystyle\operatornamewithlimits{minimize}_{l_x \leq x \leq u_x, l_z \leq y_s \leq u_z} &\;\, \sum\limits_{i\in\mathcal I} c_i \,x_i + \sum\limits_{g\in\mathcal G} c_g \,x_g + \sum\limits_{s=1}^S \left(\sum\limits_{g\in\mathcal G} p_{sg}\,x_g\right) \left[\sum\limits_{i\in\mathcal I} \sum\limits_{j\in\mathcal J} (q_{is}-\pi_{js})\,y_{ijs}\right]\\
  \mbox{subject to}  & \;\, \sum\limits_{i\in\mathcal I} c_{i}\,x_i+\sum\limits_{g\in\mathcal G} c_g\,x_g\leq B\;\,\text{(budget constraint)},\qquad\sum\limits_{g\in\mathcal G} x_g = 1,\\
  & \;\, \sum\limits_{j\in\mathcal J} y_{ijs}\leq x_{i},\quad i\in\mathcal I,\;s=1,\ldots,S\quad\text{(capacity constraints)},\\
    & \;\, \sum\limits_{i\in\mathcal I} y_{ijs} = d_{js},\quad j\in\mathcal J,\;s=1,\ldots,S\quad\text{(demand constraints)}.
\end{split}
\end{equation}

In our experiments, 
we set $|\mathcal I| = |\mathcal G| = 5$ and $|\mathcal J| = 8$. The box constraints of $x$ and $z_s$ are $[8, 15]^5 \times [0,1]^5$ and $[0,5]^{5 \times 8}$ for each $s = 1,2,\cdots,S$. The unit costs in the first-stage $\{c_i\}_{i\in\mathcal I}\cup\{c_g\}_{g\in\mathcal G}$ are independently generated from a uniform distribution on $[0, 5]$. For each scenario, $\{q_{is}\}_{i\in\mathcal I}, \{\pi_{js}\}_{j\in\mathcal J}$ and $\{d_{js}\}_{j\in\mathcal J}$ are generated from truncated normal distributions $\mathcal N(1,5^2)$ on $[2,4], [3,5]$ and $[2,5]$, respectively. To construct a set of probabilities $\displaystyle\cup_{s=1}^S\{p_{sg}\}_{g\in\mathcal G}$ satisfying $\sum_{s=1}^S p_{sg} = 1$, we first randomly generate $S\times|\mathcal G|$ values from uniform distributions on $[0,1]$, and then   group every $|\mathcal G|$ values and normalizing them such that the sum of values in each group is $1$. 
All the experiments are conducted in \textsc{Matlab} 2021b on a Intel Xeon workstation with sixteen 3.70 GHz processors and 128 GB of RAM.

\subsection{Fixed scenarios}
Since the text example with fixed scenarios is in fact a large-scale nonconvex quadratic problem, it can also be directly solved by off-the-shelf nonlinear programming solvers. We compare the performance of our proposed decomposition algorithm based on the partial Moreau envelope (DPME) with the interior-point-based solvers
Knitro \cite{byrd2006k}  and IPOPT \cite{wachter2006implementation}, both of which run with linear solver MUMPS 5.4.1. The absolute and relative feasibility and optimality errors are computed according to the termination criteria of Knitro\footnote{Knitro user guide: https://www.artelys.com/docs/knitro/2\_userGuide/termination.html}. 

The quantities $\text{KKT}_{\rm abs}$ and $\text{KKT}_{\rm rel}$ are defined as the max of absolute and relative feasibility and optimality errors, respectively. The initial points are chosen to be the same for all algorithms. Although this may not necessarily force all algorithms to converge to the same objective values, we do observe such a phenomenon in the experiments.
Further implementation details  of these algorithms are provided below.\\[0.05in]
\underline{{\bf Knitro} (version 13.0.0)}: ``{\sl knitro\_qp}" function is called in our numerical experiments to solve nonconvex quadratic programs from the \textsc{Matlab} environment. We set ``hessopt = 0" to compute the exact Hessian in the interior point method instead of using the (L)BFGS approximations, as we have  observed that the former choice is faster for all the problems tested here. We directly set ``convex = 0" to declare  our problems are nonconvex so that the solver does not need to  spend  time on checking the convexity of the problems. We report the results based on  three different settings:
    \begin{enumerate}[leftmargin=*]
\item Knitro-direct: Set ``$\text{algorithm} = 1$" so that the direct solver is used to solve linear equations. For the termination options, the KKT relative and absolute  tolerance is set to be $10^{-4}$ and $10^{-2}$, respectively, i.e., ``$\text{feastol} = \text{opttol} = 10^{-4}$'' and ``$\text{feastol\_abs} = \text{opttol\_abs} = 10^{-2}$''.
        \item Knitro-CG-1: Set ``$\text{algorithm} = 2$" so that the KKT system is solved using a projected conjugate gradient method. Stopping criteria are the same as above.
        \item Knitro-CG-2: All are the same as Knitro-CG-1 except that the KKT relative and absolute  tolerance is set to be $10^{-6}$ and $10^{-3}$, respectively.
    \end{enumerate}
    \vskip 0.05in
 \underline{{\bf IPOPT} (version 3.14.4)}: 
 Due to different scaling strategies and reformulations, the termination criteria of Knitro and IPOPT are not directly comparable.
 We set ``ipopt.tol = $5 \times 10^{-2}$" in our experiments as we find the computed solutions based on this tolerance are about the same quality as those provided by Knitro. We also set ``ipopt.hessian\_constant = `yes' " to use exact Hessian in the interior point method, and have not adopted the (L)BFGS method for the same reason as mentioned above.\\[0.05in]
\underline{{\bf DPME}}: Each master problem for the first stage  and the subproblem of the  second stage  are  convex quadratic programs, which we have called Gurobi to solve.  We compute the absolute (denoted as $\text{Feasabs}_\nu$) and relative (denoted as $\text{Feasrel}_\nu$) feasibility errors using the same way as Knitro. Let the overall objective value  at the $\nu$-th outer loop be $\text{Obj}_\nu$. We terminate our algorithm if all of the conditions below are satisfied and then compute the absolute and relative KKT errors (as defined in Knitro) at the last iteration point:
    \begin{equation}\label{eq:stopping}
      \text{Feasabs}_\nu \leq 10^{-2}, \;\,   \text{Feasrel}_\nu \leq 10^{-4} \;\,    \mbox{and }\frac{|\text{Obj}_{\nu - 1} - \text{Obj}_\nu|}{\max\{1, |\text{Obj}_{\nu - 1}|\}} \leq 10^{-4}.
    \end{equation}

 Table \ref{tab:fixed} and Figure \ref{fig:fixed} summarize the performance of different algorithms when the number of scenarios $S$ varies from $1,000$ to $120,000$ over $100$ independent replications (the sizes of the deterministic equivalent problems are listed in Table \ref{table:problem size}). For each algorithm, we report the mean and the standard deviation of the total iteration numbers, the absolute and relative KKT errors, objective values, and the wall-clock time. We also conduct experiments for $S = 500,000$ over $10$ independent replications to demonstrate the scalability of our decomposition algorithm and put the results in Table \ref{tab:fixed}. One may find that  for small-sized problems (such as when $S < 10,000$), the interior point method that is implemented by both Knitro and IPOPT can solve the problem faster than our DPME, which may be due to two reasons: one is  that  the gain of the decomposition cannot compensate for the overhead of the communication between the master problem and the subproblems; the other is that we have not used the second order information as in the interior point method. However, for the cases where $S$ is large (such as $S \geq 10,000$), the DPME is the fastest method and attains the smallest average objective values. While the computational time of IPOPT and Knitro-direct quickly exceeds the preset limit, 
 the computational time of the DPME scales approximately linearly in terms of the number of scenarios. Although the Knitro-CG-1 can solve all the problems within the time limit,  it cannot produce solutions that have similar  objective values with other methods. When we switch the solver to  the Knitro-CG-2 setting, it could  provide solutions with similar objective values as DPME but needs significantly longer  computational time.
Therefore, we conclude that the DPME for fixed scenarios can significantly reduce the computational time for solving large-scale nonconvex two-stage SPs.


\begin{figure}
\centering
  \begin{minipage}{0.55\textwidth}
    \centering
   \includegraphics[scale=0.23]{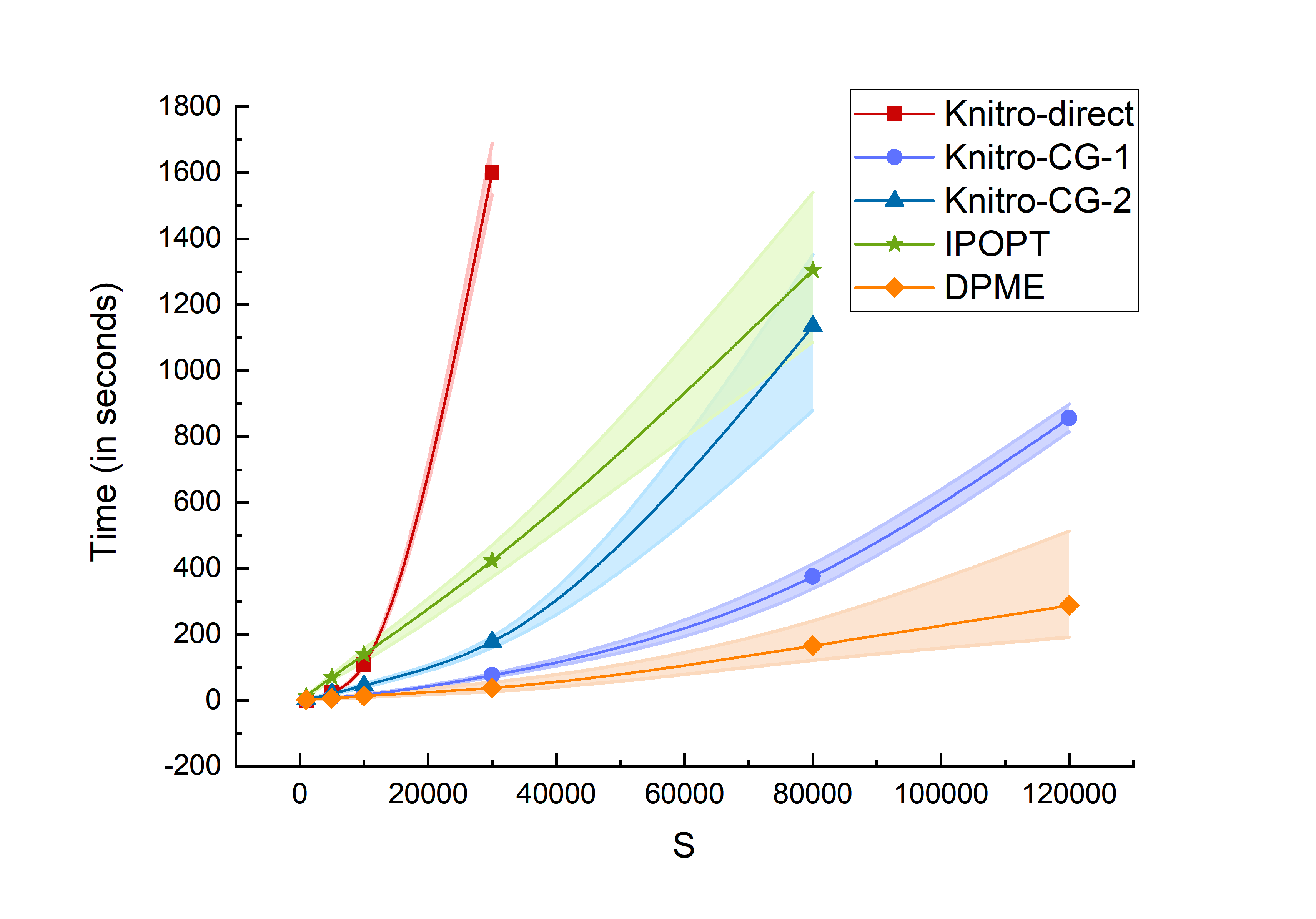}
   \vskip -0.2in
      \captionof{figure}{\scriptsize  Performance of all algorithms with different numbers of scenarios $S$ over $100$ independent replications. Shaded areas: the tubes between $10\%$ and $90\%$ quantiles of the running time; Solid lines: means of the running time.}
      \label{fig:fixed}
  \end{minipage}
  \hskip 0.12in
\begin{minipage}{0.4\linewidth}
\centering
\footnotesize
\begin{tabular}{c c c}
\toprule
\multirow{2}{*}{$S$} &\multicolumn{2}{c}{problem sizes}\\[0.05in]
\cline{2-3} \\[-0.08in]
&{rows} &{columns}\\[0.03in]
\toprule
    1,000 &93,022 &40,010\\[0.03in]
    5,000 &465,022 &200,010\\[0.03in]
    10,000 &930,022 &400,010\\[0.03in]
    30,000 &2,790,022 &1,200,010\\[0.03in]
    80,000 &7,440,022 &3,200,010\\[0.03in]
    120,000 &11,160,022 &4,800,010\\[0.03in]
    500,000 &46,500,022 &20,000,010\\[0.03in]
    \toprule
\end{tabular}
     \captionof{table}{\scriptsize Dimensions of test problems. ``Rows" stands for the number of constraints; ``Columns" stands for the number of variables.}
     \label{table:problem size}
    \end{minipage}
    \vskip -0.2in
\end{figure}

\begin{sidewaystable}
\footnotesize
\setlength\tabcolsep{3pt}
\begin{center}

\begin{tabular}{c c c c c c}
\toprule
\multirow{2}{*}{$S$} &{iterations} &{KKT$_\text{abs}$} &{KKT$_\text{rel}$} &{objective values} &{time (in seconds)}\\
 &{a $ | $ b $ | $ c $ | $ d $ | $ e}  &{a $ | $ b $ | $ c $ | $ d $ | $ e} &{a $ | $ b $ | $ c $ | $ d $ | $ e}  &{a $ | $ b $ | $ c $ | $ d $ | $ e} &{a $ | $ b $ | $ c $ | $ d $ | $ e}\\
\toprule
\multirow{2}*{1,000}
& 8$|$ 17$|$ 28$|$ 29$|$ {12}
& 3.0-4$|$ 7.4-4$|$ 4.0-6$|$ 7.4-5$|$ {1.6-3}
& 5.6-5$|$ 1.4-4$|$ 6.0-7$|$ 1.4-5$|$ {1.4-4}
&\multirow{2}{*}{93.056$|$ 93.055$|$ {\bf 93.053}$|$ 93.055$|$ {\bf 93.053}}
& {\bf 1}$|$ 2$|$ 3$|$ 13$|$ {3}\\

& (2$|$ 3$|$ 5$|$ 3$|$ {11})
& (1.6-4$|$ 1.1-3$|$ 1.5-6$|$ 1.1-5$|$ {2.1-3})
& (2.8-5$|$ 2.2-4$|$ 2.5-7$|$ 2.5-6$|$ {3.2-4})
&& (0$|$ 0$|$ 1$|$ 1$|$ {2})\\ \hline

\multirow{2}*{5,000}
& 9$|$ 16$|$ 42$|$ 31$|$ {11}
& 3.7-4$|$ 2.4-4$|$ 4.0-6$|$ 5.6-5$|$ {1.4-3}
& 6.9-5$|$ 4.9-5$|$ 7.0-7$|$ 1.1-5$|$ {1.2-4}
&\multirow{2}{*}{93.041$|$ 93.140$|$ {\bf 93.038}$|$ 93.040$|$ {\bf 93.038}}
& 24$|$ 8$|$ 21$|$ 70$|$ {\bf 7}\\

& (2$|$ 2$|$ 4$|$ 3$|$ {10})
& (1.5-4$|$ 1.0-4$|$ 1.6-6$|$ 2.2-5$|$ {1.7-3})
& (2.1-5$|$ 2.2-5$|$ 2.6-7$|$ 4.5-6$|$ {2.8-4})
&& (1$|$ 1$|$ 2$|$ 7$|$ {5})\\ \hline

\multirow{2}*{10,000}
& 10$|$ 16$|$ 43$|$ 32$|$ {11}
& 4.1-4$|$ 2.3-4$|$ 4.0-6$|$ 1.2-4$|$ {1.5-3}
& 7.5-5$|$ 4.9-5$|$ 6.0-7$|$ 2.1-5$|$ {1.3-4}
&\multirow{2}{*}{93.049$|$ 93.463$|$ {\bf 93.046}$|$ 93.059$|$ {\bf 93.046}}
& 108$|$ 17$|$ 45$|$ 139$|$ {\bf 13}\\

& (3$|$ 1$|$ 3$|$ 3$|$ {9})
& (1.2-4$|$ 1.1-4$|$ 1.5-6$|$ 6.4-4$|$ {1.8-3})
& (1.9-5$|$ 2.2-5$|$ 2.5-7$|$ 1.1-4$|$ {2.9-4})
&& (5$|$ 1$|$ 6$|$ 15$|$ {9})\\ \hline

\multirow{2}*{30,000}
& 10$|$ 16$|$ 47$|$ 35$|$ {11}
& 3.7-4$|$ 2.3-4$|$ 3.0-6$|$ 4.0-4$|$ {1.5-3}
& 6.3-5$|$ 5.0-5$|$ 5.0-7$|$ 7.2-5$|$ {1.3-4}
&\multirow{2}{*}{93.068$|$ 98.336$|$ 93.061$|$ 93.076$|$ {\bf 93.060}}
& 1600$|$ 76$|$ 178$|$ 424$|$ {\bf 38}\\

& (4$|$ 2$|$ 5$|$ 3$|$ {9})
& (1.5-4$|$ 8.3-5$|$ 1.4-6$|$ 2.1-3$|$ {1.8-3})
& (2.4-5$|$ 1.8-5$|$ 2.4-7$|$ 3.8-4$|$ {2.9-4})
&& (63$|$ 4$|$ 21$|$ 38$|$ {28})\\ \hline

\multirow{2}*{80,000}
& --$|$ 14$|$ 46$|$ 35$|$ {11}
&  --$|$ 2.2-4$|$ 3.0-6$|$ 7.7-5$|$ {1.4-3}
&  --$|$ 4.6-5$|$ 5.0-7$|$ 1.6-5$|$ {1.2-4}
&\multirow{2}{*}{--$|$ 122.644$|$ {\bf 93.078}$|$ 93.127$|$ {\bf 93.078}}
& {\bf t}$|$ 376$|$ 1135$|$ 1305$|$ {\bf 166}\\

& (--$|$ 4$|$ 8$|$ 4$|$ {9})
& (--$|$ 6.6-5$|$ 1.3-6$|$ 3.8-4$|$ {1.8-3})
& (--$|$ 1.4-5$|$ 2.3-7$|$ 8.1-5$|$ {2.8-4})
&& (--$|$ 28$|$ 181$|$ 177$|$ {121})\\ \hline

\multirow{2}*{120,000}
&  --$|$ 11$|$ --$|$ --$|$ {12}
&  --$|$ 2.4-4$|$ --$|$ --$|$ {1.5-3}
&  --$|$ 5.1-5$|$ --$|$ --$|$ {1.4-4}
&\multirow{2}{*}{--$|$ 133.885$|$ --$|$ --$|$ {\bf 93.101}}
&  {\bf t}$|$ 856$|$ {\bf t}$|$ {\bf t}$|$ {\bf 288}\\

&  (--$|$ 3$|$ --$|$ --$|$ {11})
&  (--$|$ 8.3-5$|$ --$|$ --$|$ {1.8-3})
&  (--$|$ 1.7-5$|$ --$|$ --$|$ {3.0-4})
&&  (--$|$ 40$|$ --$|$ --$|$ {246}) \\ \hline

\multirow{2}*{500,000}
& --$|$ 11$|$ --$|$ --$|$ {14}
& --$|$ 2.0-4$|$ --$|$ --$|$ {1.5-3}
& --$|$ 4.1-5$|$ --$|$ --$|$ {1.3-4}
&\multirow{2}{*}{--$|$ 141.974$|$ --$|$ --$|$ {\bf 96.485}}
& {\bf t}$|$ 17464$|$ {\bf t}$|$ {\bf t}$|$ {\bf 1681}\\

&  (--$|$ 2$|$ --$|$ --$|$ {13})
&  (--$|$ 8.8-7$|$ --$|$ --$|$ {8.3-4})
&  (--$|$ 3.0-6$|$ --$|$ --$|$ {1.6-4})
&& (--$|$ 553$|$ --$|$ --$|$ {1571}) \\
\toprule
\end{tabular}
\caption{\scriptsize The performance of Knitro-direct, Knitro-CG, IPOPT, and DPME. In the table, ``$S$'' is the number of scenarios; the numbers without parentheses are the means over $100$ replications and the  numbers in parentheses are the standard deviations;  ``t" means the method exceeds the preset time limit, which is $1,800$s for $S \leq 120,000$ and $18,000$s for $S = 500,000$; ``$a$" stands for Knitro-direct; ``$b$" stands for Knitro-CG-1; ``$c$" stands for Knitro-CG-2; ``$d$" stands for IPOPT and ``$e$" stands for DPME.}
\label{tab:fixed}

\begin{tabular}{ c c c c c c c}
\toprule
\multirow{2}{*}{$\sigma$} & \multirow{2}{*}{$\eta$} & {iterations} & \multirow{2}{*}{KKT$_\text{abs}$} & \multirow{2}{*}{KKT$_\text{rel}$} & \multirow{2}{*}{objective values} & \multirow{2}{*}{time (s)}\\
&& {outer $|$ total} &&&&\\
\toprule

\multirow{3}*{0.5}

& 100
& 2 (0) $|$ 6 (5)
& 4.3-3 (1.7-2)
& 1.1-3 (4.6-3)
& 96.224
& {26 (13)}\\

& 200
& 2 (0) $|$ 6 (5)
& 3.0-3 (1.3-2)
& 8.4-4 (3.8-3)
& 96.221
& {26 (16)}\\

& 400
& 2 (0) $|$ {6 (7)}
& 2.2-3 (1.1-2)
& 5.1-4 (2.6-3)
& 96.219
& {32 (27)}\\

& 800
& 2 (0) $|$ 6 (5)
& 7.3-4 (4.3-3)
& 1.6-4 (9.1-4)
& 96.216
& {32 (18)}\\

& --
& 2 {(0)} $|$ 6 (5)
& 3.8-4 (2.2-3)
& 9.1-5 (5.1-4)
& 96.215
& {42 (27)}\\ \hline

\multirow{3}*{2}

& 100
& 2 (0) $|$ 8 (11)
& {3.8-2} (1.1-1)
& {7.5-3 (2.2-2)}
& {93.458}
& {32 (29)}\\

& 200
& 2 (0) $|$ 7 (8)
& {8.2-3} (3.7-2)
& 1.5-3 (6.7-3)
& 93.406
& {31 (26)}\\

& 400
& 2 (0) $|$ 7 (8)
& {4.1-3} (1.6-2)
& {7.5-4} (2.8-3)
& {93.400}
& {32 (26)}\\

& 800
& 2 (0) $|$ 6 {(6)}
& {2.7-3} (1.3-2)
& {5.0-4} (2.4-3)
& {93.400}
& {34 (23)}\\

& 1600
& 2 (0) $|$ 6 (5)
& {2.0-3} (8.6-3)
& {3.7-4} (1.6-3)
& 93.397
& {32 (19)}\\

& --
& 2 (0) $|$ 6 (5)
& {5.1-4 (1.9-3)}
& {8.8-5 (3.1-4)}
& 93.395
& {42 (27)}\\ \hline

\multirow{3}*{5}

& 100
& {3} (0) $|$ {15 (18)}
& {1.2-1} (2.0-1)
& {1.8-2 (3.1-2)}
& {92.444}
& {48} (46)\\

& 400
& 2 (0) $|$ {13 (16)}
& {3.2-2 (5.9-2)}
& {4.7-3 (9.4-3)}
& {92.289}
& {56 (58)}\\

& 1600
& 2 (0) $|$ {11 (15)}
& {2.4-2 (3.7-2)}
& {3.3-3 (5.2-3)}
& {92.271}
& {53 (56)}\\

& 3200
& 2 (0) $|$ {11 (15)}
& {2.2-2 (3.7-2)}
& {3.1-3 (5.1-3)}
& {92.268}
& {53 (59)}\\

& 6400
& 2 (0) $|$ {11 (15)}
& {2.1-2 (3.7-2)}
& {2.9-3 (5.1-3)}
& {92.266}
& {58 (66.1)}\\


& --
& 2 (0) $|$ {11 (15)}
& {2.1-2 (3.7-2)}
& {2.9-3 (5.0-3)}
& {92.266}
& {70 (82)}\\
\toprule
\end{tabular}
\caption{\footnotesize The performance of the sampling-based DPME. In the table, ``$\sigma$" is the variance of the normal distribution from which we generate the data; ``$\eta$'' represents the linear growth rate of sample size such that the number of scenarios used in $\nu$-th outer iteration $S_\nu = \eta \nu$; ``--'' in the column of $\eta$ stands for the benchmark of DPME using full scenarios.}
\label{tab:sampling}
\end{center}
\end{sidewaystable}

\subsection{Sampling-based decomposition}

We test the sampling-based DPME proposed in  Algorithm \ref{alg:sampling_based} for the same test problem  with the total number of scenarios $S = 50,000$. Instead of using all scenarios at each iteration, we gradually add them to reduce the computational cost especially at the early stage. In our experiments, the sample size $S_\nu$ is taken as $ \eta \nu$ for different positive parameter $\eta$.
In order to understand how the growth rate $\eta$ depends on the distributions of the random scenarios, we generate  $\{q_{is}\}_{i\in\mathcal I},\{\pi_{js}\}_{j\in\mathcal J}$ and $\{d_{js}\}_{j\in\mathcal J}$ from truncated normal distributions $\mathcal N(1,\sigma^2)$ on $[2,13], [3,10]$ and $[2,5]$ with $\sigma \in \{0.5, 2, 5\}$, and varies the values of $\eta$. 

The stopping criteria for the sampling-based DPME is the same as \eqref{eq:stopping}, where we check the violation of the KKT system for the deterministic equivalent problem formulated by  all scenarios $S = 50,000$. However, unlike the case for fixed scenarios, we do not have all the second-stage solutions $\{y_s\}_{s=1}^S$ to compute the KKT residual since some samples may not have been used yet. To {resolve} this issue, we compute all  $\{y_s\}_{s=1}^S$  at every $\nu$-th outer iteration, and then estimate the  multipliers corresponding to  the first-stage budget and box constraints by  minimizing the current KKT residual.

In Table \ref{tab:sampling}, we summarize the performance of Algorithm \ref{alg:sampling_based} for different combinations of $(\sigma, \eta)$, where we also provide the results obtained from Algorithm \ref{alg:finite_samples} without sampling for benchmarks. It can be observed from the table that problems with larger variability may need faster growth rate of the batch size to retain the same level of solution quality. If the growth rate is properly chosen, the sampling-based DPME can outperform the fixed-scenario version in the computational time with comparable solution qualities.

\section{Conclusion}
Compared with the extensive literature on the algorithms for convex (especially linear) two-stage SPs, efficient computational algorithms for solving continuous nonconvex two-stage SPs have been much less explored.
In this paper, we have made a first attempt on developing the decomposition scheme for a special class of latter problems.
The key of the proposed algorithm is the derivation of successive strongly convex approximations of the nonconvex recourse functions. We hope the work done in paper could stimulate researchers' interests in a broader paradigm of SPs that goes beyond the classical convex settings. There are a lot of open questions that deserve future investigations, such as 
 how to combine the stochastic dual dynamic programming approach with the  tools developed in the current paper  to solve nonconvex  multistage  SPs, 
as well as  how to design rigorous stopping criteria for the general nonconvex SPs with continuous distributions.

\bibliography{reference}

\end{document}